\documentclass{amsart}

\usepackage{amsmath,amssymb,amsthm}
\usepackage{mathrsfs}
\usepackage{booktabs}
\usepackage{hyperref}
\usepackage{geometry}
\usepackage{graphicx}
\usepackage{enumitem}
\usepackage{cleveref}

\newtheorem{theorem}{Theorem}[section]
\newtheorem{corollary}[theorem]{Corollary}
\newtheorem{lemma}[theorem]{Lemma}
\newtheorem{proposition}[theorem]{Proposition}
\theoremstyle{definition}
\newtheorem{definition}[theorem]{Definition}
\newtheorem{remark}[theorem]{Remark}

\crefname{theorem}{Theorem}{Theorems}
\crefname{corollary}{Corollary}{Corollaries}
\crefname{lemma}{Lemma}{Lemmas}
\crefname{proposition}{Proposition}{Propositions}
\crefname{definition}{Definition}{Definitions}
\crefname{remark}{Remark}{Remarks}
\crefname{example}{Example}{Examples}
\crefname{equation}{Eq.}{Eqs.}

\title{
Spectral Purification in Reversible Markov Chains: Hidden Parameters, Observable Equivalence, and Finite-Time Rigidity
}

\author{Qiao Wang}
\address{School of Information Science and Engineering, and School of Economics and Management, Southeaast University, Nanjing, 211189, China}

\email{qiaowang@seu.edu.cn}
\thanks{This research was conducted purely out of the author’s independent academic interest in the mathematical
foundations of the subject and received no financial support from any funding agency, grant, or institutional
programme.}

\date{}



\begin{document}

\begin{abstract}

Classical spectral theory of reversible Markov chains describes the asymptotic outcome of relaxation: the slowest nontrivial eigenmode ultimately dominates, at a rate governed by the spectral gap. This paper studies a different, essentially finite-time question. Long before a chain approaches stationarity, many observables of the relaxation trajectory already behave as though only a single spectral mode were present, as if the underlying multi-mode structure had already collapsed. We call this phenomenon \emph{spectral purification} and develop a quantitative theory of it.

The theory proceeds through three structural discoveries. The first is a change of primary object: purification is properly studied not through the state-space trajectory itself, but through the modal distribution $\{p_i(k)\}$, the probability vector obtained by normalizing spectral energy across modes. The second is a hidden purification parameter $x_k$, built from the spectral separation ratio $\lambda_3/\lambda_2$, that governs the entire evolution of $\{p_i(k)\}$. The third is an equivalence class of observables: we prove that six \textit{a priori} distinct diagnostics of purification (the slow-mode energy fraction, the power-iteration error, the direction deficit, the Rayleigh quotient error, the eigenvalue estimate error, and the spectral entropy) all reduce to $x_k$ to leading order, differing only by an explicit, computable constant.

This equivalence class yields sharp, non-asymptotic two-sided bounds on the rigidity time $T_{\mathrm{rigid}}(\delta)$ at which the slowest mode captures a prescribed fraction of the spectral energy --- governed by the spectral ratio $\lambda_3/\lambda_2$ rather than the spectral gap $1-\lambda_2$ --- and an exact, non-asymptotic entropy representation of the same process, including a spectral Clausius equality and a spectral second law $G(k+1)\le G(k)$. We further identify the boundary of the theory: the asymptotic variance of the classical time-average estimator is monotone in $\lambda_3/\lambda_2$, but its finite-$k$ mean-squared-error correction does not belong to the exponential regime governed by $x_k$.

Applied to power iteration, the theory yields an exact error identity, an observable spectral variance formula, and a fully data-driven adaptive stopping criterion with provable guarantees.

\end{abstract}

\maketitle
\section{Introduction}

For a reversible Markov chain, classical spectral theory decomposes any centered observable as $g_k = P^k g_0 = \sum_{i\ge2} c_i \lambda_i^k \phi_i$, where $1=\lambda_1>\lambda_2\ge\cdots\ge\lambda_n\ge-1$ are the eigenvalues of $P$ and $\{\phi_i\}$ the corresponding orthonormal basis on $L^2(\pi)$, and it identifies the asymptotic winner of this decomposition: the slowest eigenmode $\phi_2$, with $g_k/\|g_k\|_\pi \to \pm\phi_2$ at a rate governed by the spectral gap $1-\lambda_2$. This is completely understood.

But if one actually tracks the trajectory of a typical observable, rather than only its limit, a puzzling fact appears. Long before the chain approaches stationarity --- often well before the mixing time --- many observables built from $g_k$ already behave as though only a single mode were present: convergence becomes purely exponential, direction stabilizes, and statistics stop reflecting the underlying multi-mode structure. Asymptotic theory has no mechanism for this: mixing-time bounds track only the total energy $E_k=\|g_k\|_\pi^2$, which is blind to how that energy is distributed among modes --- and it is exactly that distribution which appears to have already collapsed. This is a physical question before it is a mathematical one, and we make it precise through four more specific ones.

\textit{Question 1 (what happens).} Long before a reversible Markov chain approaches stationarity, many of its observables already exhibit near one-dimensional dynamics. What, precisely, is collapsing?

\textit{Question 2 (a lower-dimensional description).} Can this collapse be captured by a description of lower dimension than the full spectral decomposition --- rather than merely observed as a qualitative trend?

\textit{Question 3 (the governing parameter).} If so, what is the single quantity that governs it?

\textit{Question 4 (unity of observables).} Do the different quantities one could use to detect this collapse --- convergence error, direction, entropy, and so on --- actually measure the same underlying process, or six different ones wearing one name?

A fifth question follows once the first four are answered: precisely \emph{when} does the collapse occur, in explicit, non-asymptotic terms? Rigidity theory, taken up in \cref{sec:rigidity}, does not introduce a new mechanism to answer this; it quantifies, using the parameter already identified in Question 3, the time scale on which the purification process becomes visible.

We call the phenomenon of Question 1 \emph{spectral purification}: the dynamical concentration of spectral weight onto the slowest surviving mode; \cref{sec:purification} makes this object precise. The central discovery of the theory, and the answer to Questions 2 and 3, is that this entire process admits a one-dimensional description: \cref{sec:hidden-parameter} shows that spectral purification is governed, in its entirety, by a single hidden parameter $x_k$, built from the spectral separation ratio $\lambda_3/\lambda_2$ --- not one ingredient among several, but the coordinate in which the whole process is recorded. \Cref{sec:observable-equivalence} answers Question 4: six \textit{a priori} distinct observables of purification all reduce to $x_k$, up to an explicit constant, so the answer is one process, not six. \Cref{sec:rigidity} then answers the deferred fifth question, with sharp, non-asymptotic two-sided bounds on the rigidity time $T_{\mathrm{rigid}}(\delta)$. \Cref{sec:thermodynamics} shows that entropy provides yet another language for the same process, and \cref{sec:boundary} identifies what does \emph{not} belong to this theory: a classical estimation-variance question that $x_k$ influences but does not fully explain.

Answering these questions requires a change of primary object. The state-space trajectory $g_k$ carries the full vector, but the only thing classical theory extracts from it is the total energy $E_k=\|g_k\|_\pi^2$, and total energy says nothing about \emph{how} that energy is shared among modes --- which is exactly what is collapsing. What purification requires instead is the object obtained by normalizing the modal energies into proportions: the modal distribution $\{p_i(k)\}$, a probability vector over the surviving modes. This is not a change of notation but a change of primary object, from a single vector evolving in $\mathbb{R}^n$ to a probability distribution evolving over a finite set of modes --- the same move, in spirit, that replaces the trajectory of a single particle with the evolution of a probability distribution over its accessible states. The remainder of this paper is organized around the evolution of $\{p_i(k)\}$, not $g_k$ itself, because purification concerns the redistribution of spectral weight, not the decay of total energy.

\subsection{What is classical and what is new}

The mathematical ingredients are classical. Self-adjoint spectral decomposition, orthogonal eigenbases, and Dirichlet-form methods are standard~\cite{levin2009,aldous2002,saad2003}. The identities
\[
E_{k+1} = \langle g_k,P^2g_k\rangle_\pi, \qquad P^2 = I-2\mathcal{G}+\mathcal{G}^2
\]
are elementary consequences of reversibility.

What is new is the organization of these ingredients into a finite-time spectral dynamical theory. We show that spectral relaxation admits two distinct scales: a global mixing scale governed by \(1-\lambda_2\), and a dominance-formation scale governed by \(\lambda_3/\lambda_2\). These scales are generally distinct. The theory develops in four steps, following the same chain established in the introduction: hidden parameter, then observable equivalence, then rigidity, then entropy and its algorithmic consequences.

\textit{Step 1: the hidden purification parameter.}
We identify a single scalar $x_k$, built from the spectral ratio $\lambda_3/\lambda_2$, that governs to leading order the entire purification process (\cref{sec:hidden-parameter}).

\textit{Step 2: observable equivalence.}
We show that \emph{every} natural observable of purification --- the slow-mode fraction, the power-iteration error, the direction deficit, the Rayleigh quotient error, the eigenvalue estimate error, and the spectral entropy --- reduces to $x_k$ to leading order, up to an explicit constant (Theorem~\ref{thm:equivalence-class}).

\textit{Step 3: finite-time rigidity.}
We introduce \(T_{\mathrm{rigid}}(\delta)\) and derive sharp two-sided bounds (Theorem~\ref{thm:sharp-rigidity}), showing that collapse is governed by \(\lambda_3/\lambda_2\), not by the spectral gap.

\textit{Step 4: entropy representation and algorithmic consequences.}
We derive the exact entropy identity
\[
\Delta S
=
\frac{\mathrm{Cov}}{\rho}
-
D_{\mathrm{KL}},
\]
and a canonical covariance representation (Proposition~\ref{prop:canonical-covariance}). In the two-mode case, covariance changes sign at \(T_{\mathrm{rigid}}(1/2)\), where entropy reaches \(\log 2\). For general chains, we establish a rigidity threshold (Theorem~\ref{thm:general-rigidity-threshold}) and a spectral second-law functional \(G(k)=E_kS_{\mathrm{spec}}(k)\) satisfying \(G(k+1)\le G(k)\) (Theorem~\ref{thm:G-monotonicity}). Applied to power iteration, the theory yields an exact error identity, a spectral variance formula
\[
\operatorname{Var}_{p_k}(\lambda^2)
=
\rho_k(\rho_{k+1}-\rho_k),
\]
and an adaptive stopping criterion with provable guarantees (Theorem~\ref{thm:stopping-criterion}).

We also identify a boundary of the theory: the asymptotic variance of the classical time-average estimator is monotone in the spectral ratio $\rho_\star=\lambda_3/\lambda_2$ (Corollary~\ref{cor:rho-variance}), but its finite-$k$ correction does not belong to the exponential regime governed by $x_k$ (Remark~\ref{rem:mse-incommensurable}).

The contribution is not a new asymptotic theorem. It is a finite-time theory explaining how a multi-mode relaxation process progressively loses spectral complexity and becomes effectively one-dimensional --- and doing so through a single scalar parameter that every observable of the process can be read off from.

\subsection{Relation to existing literature}

\textit{Classical Markov chain theory.}
Spectral decomposition and asymptotic dominance are classical~\cite{levin2009}. The present theory is complementary: it supplies finite-time criteria for the onset of that dominance, a question classical theory does not address.

\textit{Numerical linear algebra.}
Power iteration and acceleration methods~\cite{saad2003,trefethen2013,polyak1964} are reinterpreted here as instances of optimal spectral shaping.

\textit{Stochastic thermodynamics.}
We extend entropy-production ideas~\cite{esposito2010,seifert2012} to spectral space; the resulting formalism is complementary to configuration-space stochastic thermodynamics rather than a generalization of it.

\textit{Spectral entropy dynamics.}
We decompose spectral entropy into a slow-fast competition term and an internal-disorder term, which together account for its finite-time evolution.

\textit{Metastability and first-passage phenomena.}
Rigidity emergence provides quantitative estimates for the onset of exponential first-passage statistics~\cite{hartich2018}, developed in \cref{app:first-passage}.

\textit{Cutoff phenomena.}
The spectral entropy collapse studied here is complementary to, and may offer an alternative diagnostic for, the classical cutoff phenomenon~\cite{diaconis1987}.

\subsection{Organization of the paper}

Section~\ref{sec:relaxation} introduces the relaxation framework, and Section~\ref{sec:dissipation} establishes the exact spectral dissipation laws. Section~\ref{sec:purification} identifies the modal distribution as the central object of study and fixes the terminology of rigidity. Section~\ref{sec:hidden-parameter} introduces the hidden purification parameter $x_k$, and Section~\ref{sec:observable-equivalence} shows that every natural purification observable reduces to $x_k$. Section~\ref{sec:rigidity} converts this into sharp finite-time rigidity bounds, and Section~\ref{sec:thermodynamics} develops the exact entropy dynamics of purification. Section~\ref{sec:boundary} identifies the limits of the theory. Sections~\ref{sec:variational}--\ref{sec:examples} apply the framework to spectral acceleration, power iteration, and numerical examples, and Section~\ref{sec:discussion} concludes.
\section{Reversible Relaxation Systems}
\label{sec:relaxation}

\subsection{Reversible Markov chains}

Let $\Omega = \{1,\dots,n\}$ be a finite state space and let
\[
P : \Omega \times \Omega \to [0,1]
\]
be an irreducible Markov kernel with invariant probability measure $\pi$ satisfying
\[
\sum_{x\in\Omega} \pi(x) P(x,y) = \pi(y).
\]
Throughout the paper we assume that $P$ is reversible with respect to $\pi$, namely
\begin{equation}\label{eq:detailed-balance}
\pi(x) P(x,y) = \pi(y) P(y,x), \qquad \forall\, x,y \in \Omega.
\end{equation}

We equip the space of real-valued functions on $\Omega$ with the Hilbert structure
\[
\langle f, g \rangle_\pi := \sum_{x\in\Omega} f(x) g(x) \pi(x),
\qquad
\|f\|_\pi^2 = \langle f, f \rangle_\pi.
\]
Under the reversibility condition \eqref{eq:detailed-balance}, the Markov operator
\[
(Pf)(x) = \sum_{y\in\Omega} P(x,y) f(y)
\]
is self-adjoint on $L^2(\pi)$.

By the spectral theorem for self-adjoint operators on finite-dimensional Hilbert spaces, there exists an orthonormal basis $\{\phi_1,\dots,\phi_n\}$ of $L^2(\pi)$ consisting of eigenvectors of $P$:
\[
P\phi_i = \lambda_i \phi_i, \qquad \langle \phi_i, \phi_j \rangle_\pi = \delta_{ij}.
\]
We order the eigenvalues as
\begin{equation}\label{eq:spectrum-order}
1 = \lambda_1 > \lambda_2 \ge \lambda_3 \ge \cdots \ge \lambda_n \ge -1.
\end{equation}
Irreducibility implies that the eigenspace corresponding to $\lambda_1 = 1$ is one-dimensional and generated by the constant function
\[
\phi_1 = \mathbf{1}, \qquad \mathbf{1}(x) \equiv 1.
\]

\subsection{The relaxation operator}

The central object of this paper is the \emph{relaxation operator}
\begin{equation}\label{eq:relaxation-operator}
\mathcal{G} := I - P.
\end{equation}

The operator $\mathcal{G}$ is not new. It appears as the generator of the Dirichlet form
\[
\mathcal{E}(f,f) = \langle f, \mathcal{G} f \rangle_\pi
= \frac12 \sum_{x,y\in\Omega} \pi(x) P(x,y) \bigl(f(x) - f(y)\bigr)^2,
\]
which is the standard tool for variational characterizations of the spectral gap~\cite{levin2009}. In that context, $\mathcal{G}$ is used to prove inequalities---Poincar\'e, log-Sobolev, Nash---that yield upper bounds on mixing times. Our use of $\mathcal{G}$ is different: we treat it as the \emph{primary dynamical object} and study the exact, finite-time identities that it generates, rather than the asymptotic estimates that can be derived from it.

Because $P$ is self-adjoint, $\mathcal{G}$ is also self-adjoint on $L^2(\pi)$. The Dirichlet form computation shows that $\mathcal{G}$ is positive semidefinite. Its spectrum follows directly from \eqref{eq:spectrum-order}: if $P\phi_i = \lambda_i \phi_i$, then
\[
\mathcal{G}\phi_i = (1 - \lambda_i) \phi_i.
\]
We therefore define the \emph{relaxation spectrum}
\begin{equation}\label{eq:mu-def}
\mu_i := 1 - \lambda_i, \qquad i = 1,\dots,n.
\end{equation}
Under the ordering \eqref{eq:spectrum-order},
\begin{equation}\label{eq:mu-order}
0 = \mu_1 < \mu_2 \le \mu_3 \le \cdots \le \mu_n \le 2.
\end{equation}

The quantity $\mu_2 = 1 - \lambda_2$ is the classical spectral gap. In the standard theory, $\mu_2^{-1}$ is the relaxation time, and it controls the asymptotic exponential decay rate of $E_k$. In our framework, the \emph{entire} relaxation spectrum $\{\mu_i\}_{i=2}^n$ acquires direct dynamical significance: each $\mu_i$ is the dissipation rate of the $i$-th eigenmode, and the higher-order spectral separations $\mu_3/\mu_2, \mu_4/\mu_3, \dots$ govern the finite-time structure of spectral purification.

\subsection{Relaxation trajectories}

\begin{definition}[Relaxation trajectory]\label{def:trajectory}
Let $g_0 \in L^2(\pi)$ satisfy the centering condition $\langle g_0, \mathbf{1} \rangle_\pi = 0$. The \emph{relaxation trajectory} generated by $g_0$ is the sequence
\begin{equation}\label{eq:trajectory-def}
g_k := P^k g_0, \qquad k = 0,1,2,\dots
\end{equation}
The \emph{energy} of the trajectory at step $k$ is
\begin{equation}\label{eq:energy-def}
E_k := \|g_k\|_\pi^2.
\end{equation}
\end{definition}

The centering condition removes the stationary mode ($\lambda_1 = 1$, $\mu_1 = 0$), which does not dissipate. As a consequence, $g_k \to 0$ as $k \to \infty$.

Expanding $g_0$ in the eigenbasis $\{\phi_i\}_{i=2}^n$ of the nontrivial modes,
\[
g_0 = \sum_{i=2}^n c_i \phi_i, \qquad c_i = \langle g_0, \phi_i \rangle_\pi,
\]
we obtain the spectral representation of the entire trajectory:
\begin{equation}\label{eq:spectral-expansion}
g_k = \sum_{i=2}^n c_i \lambda_i^k \phi_i.
\end{equation}

The energy admits the corresponding exact decomposition
\begin{equation}\label{eq:energy-decomposition}
E_k = \sum_{i=2}^n |c_i|^2 \lambda_i^{2k}
     = \sum_{i=2}^n n_i(k),
\end{equation}
where we have introduced the \emph{modal energies}
\begin{equation}\label{eq:modal-energy-def}
n_i(k) := |c_i|^2 \lambda_i^{2k}, \qquad i = 2,\dots,n.
\end{equation}

The developments of this paper are organized around the evolution of these modal energies --- or more precisely their normalized form, the modal distribution $\{p_i(k)\}$ introduced in \cref{sec:purification} --- rather than around the state-space trajectory $g_k$ itself. This shift of primary object, from a single vector to a probability distribution over modes, is what makes the finite-time phenomena of the following sections visible in the first place.

\section{Exact Spectral Dissipation}
\label{sec:dissipation}

\subsection{The exact dissipation identity}

\begin{theorem}[Exact spectral dissipation law]\label{thm:exact-dissipation}
For every reversible Markov chain and every centered initial condition $g_0 \perp \mathbf{1}$,
\begin{equation}\label{eq:exact-dissipation}
E_k - E_{k+1} = \bigl\langle g_k, (2\mathcal{G} - \mathcal{G}^2) g_k \bigr\rangle_\pi,
\qquad \forall\, k \ge 0.
\end{equation}
\end{theorem}

\begin{proof}
Using $g_{k+1} = P g_k$ and the self-adjointness of $P$,
\[
E_{k+1} = \langle P g_k, P g_k \rangle_\pi = \langle g_k, P^2 g_k \rangle_\pi.
\]
Substituting $P = I - \mathcal{G}$,
\[
P^2 = (I - \mathcal{G})^2 = I - 2\mathcal{G} + \mathcal{G}^2.
\]
Therefore
\[
E_{k+1} = \langle g_k, g_k \rangle_\pi - \bigl\langle g_k, (2\mathcal{G} - \mathcal{G}^2) g_k \bigr\rangle_\pi
       = E_k - \bigl\langle g_k, (2\mathcal{G} - \mathcal{G}^2) g_k \bigr\rangle_\pi,
\]
which rearranges to \eqref{eq:exact-dissipation}.
\end{proof}

\begin{remark}\label{rem:dissipation-classical}
This identity is an immediate consequence of self-adjointness and has appeared implicitly in various forms in the Dirichlet form literature and in numerical linear algebra. Its significance here is not its algebraic novelty but its role as the \emph{exact, non-asymptotic} foundation for all subsequent developments. Classical theory typically uses this identity to derive the upper bound $E_k - E_{k+1} \ge \mu_2 E_k$ (the Poincar\'e inequality), whereas we will work directly with the identity to extract finer mode-by-mode information.
\end{remark}

\subsection{Modewise dissipation decomposition}

The exact identity \eqref{eq:exact-dissipation} becomes transparent after spectral expansion.

\begin{corollary}[Modewise dissipation]\label{cor:modewise}
Let $g_k = \sum_{i=2}^n c_i \lambda_i^k \phi_i$. Then
\begin{equation}\label{eq:modewise-dissipation}
E_k - E_{k+1} = \sum_{i=2}^n (1 - \lambda_i^2) \, n_i(k).
\end{equation}
Equivalently, in terms of the relaxation spectrum $\mu_i = 1 - \lambda_i$,
\begin{equation}\label{eq:modewise-mu}
E_k - E_{k+1} = \sum_{i=2}^n (2\mu_i - \mu_i^2) \, n_i(k),
\end{equation}
where $n_i(k) = |c_i|^2 \lambda_i^{2k}$.
\end{corollary}

\begin{proof}
Since $\mathcal{G} \phi_i = \mu_i \phi_i$, we have $(2\mathcal{G} - \mathcal{G}^2) \phi_i = (2\mu_i - \mu_i^2) \phi_i$. Substituting the spectral expansion of $g_k$ into \eqref{eq:exact-dissipation} and using orthonormality of the eigenbasis,
\[
E_k - E_{k+1} = \sum_{i=2}^n (2\mu_i - \mu_i^2) |c_i|^2 \lambda_i^{2k}
              = \sum_{i=2}^n (2\mu_i - \mu_i^2) \, n_i(k).
\]
Using $1 - \lambda_i^2 = 2\mu_i - \mu_i^2$ yields \eqref{eq:modewise-dissipation}.
\end{proof}

\Cref{cor:modewise} reveals a fundamental structural property of reversible relaxation: \textbf{the dissipation decomposes exactly into independent modal contributions, with no cross terms}. Each mode $i$ dissipates energy at its own intrinsic rate $1 - \lambda_i^2 = 2\mu_i - \mu_i^2$, weighted by its current modal energy $n_i(k)$. This is the spectral analogue of the fact that, in a non-interacting multi-component system, each component relaxes independently.

This independence is what makes the subsequent analysis possible. Because the modal energies $\{n_i(k)\}$ evolve autonomously ($n_i(k+1) = \lambda_i^2 n_i(k)$), the trajectory of the modal distribution $\{p_i(k)\}$ is a closed dynamical system that can be studied in its own right.

\subsection{The relative dissipation rate}

\begin{definition}[Relative dissipation rate]\label{def:rho}
The \emph{relative dissipation rate} at step $k$ is
\begin{equation}\label{eq:rho-def}
\rho_k := \frac{E_{k+1}}{E_k} \in (0,1).
\end{equation}
Equivalently, the fraction of energy dissipated in step $k$ is $d_k := 1 - \rho_k$.
\end{definition}

From the spectral expansion,
\begin{equation}\label{eq:rho-spectral}
\rho_k = \frac{\sum_{i=2}^n n_i(k) \lambda_i^2}{\sum_{i=2}^n n_i(k)}
       = \sum_{i=2}^n p_i(k) \lambda_i^2
       = \mathbb{E}_{p_k}[\lambda^2],
\end{equation}
where $\mathbb{E}_{p_k}$ denotes expectation with respect to the modal distribution $p_k$. Thus $\rho_k$ is the $p_k$-average of the squared eigenvalues $\lambda_i^2$. This simple observation has far-reaching consequences: the observable scalar sequence $\{\rho_k\}$ carries information about the entire spectral distribution $\{p_i(k)\}$, and its fluctuations across time steps encode the distribution's variance and higher moments. We exploit this fully in \cref{sec:power-iteration}.

\begin{corollary}[Dissipation rate as expectation]\label{cor:dissipation-expectation}
The fraction of energy dissipated at step $k$ is
\[
d_k = 1 - \rho_k = \mathbb{E}_{p_k}[1 - \lambda^2] = \sum_{i=2}^n p_i(k) (2\mu_i - \mu_i^2).
\]
\end{corollary}

This completes the foundational layer. The exact dissipation identity and its modewise decomposition provide the platform on which the rigidity theory (\cref{sec:rigidity}), spectral thermodynamics (\cref{sec:thermodynamics}), and power iteration analysis (\cref{sec:power-iteration}) are built.

\section{Spectral Purification}
\label{sec:purification}

Spectral purification is not a property of a particular observable. It is a dynamical redistribution of spectral weight among eigenmodes as $k$ increases. The object of study throughout the remainder of this paper is therefore not any single quantity derived from the trajectory $g_k$, but the modal distribution $\{p_i(k)\}_{i=2}^n$ itself, together with the elementary vocabulary needed to describe its limiting behavior.

\Cref{eq:energy-decomposition} is the starting point for all subsequent developments. It expresses the total energy as a sum of independent spectral contributions, each decaying exponentially at its own characteristic rate $\lambda_i^2$. The \emph{distribution} of energy among these modes,
\begin{equation}\label{eq:modal-distribution-def}
p_i(k) := \frac{n_i(k)}{E_k}, \qquad i = 2,\dots,n,
\end{equation}
defines a probability vector on $\{2,\dots,n\}$ that evolves as the trajectory relaxes. The central concern of this paper is the evolution of $\{p_i(k)\}$---its dissipation, its entropy, its convergence to a point mass at the slowest mode---viewed through exact identities rather than asymptotic estimates.

\begin{remark}[Relation to classical formulations]\label{rem:classical-relation}
The material in this section is entirely classical. The eigenbasis expansion \eqref{eq:spectral-expansion}, the energy decomposition \eqref{eq:energy-decomposition}, and the Dirichlet form representation of $\mathcal{G}$ appear in standard textbooks~\cite{levin2009,aldous2002,saloff-coste1997}. What is non-standard is the decision to organize the subsequent analysis around $\mathcal{G}$ and the modal distribution $\{p_i(k)\}$ as primary objects, rather than around bounds on $E_k$. This shift of perspective---from estimating the total energy to tracking the internal spectral distribution---is what enables the quantitative rigidity and thermodynamic results of \Cref{sec:rigidity,sec:thermodynamics,sec:power-iteration}.
\end{remark}

The modewise dissipation decomposition (\cref{cor:modewise}) reveals that each spectral mode dissipates independently. We now investigate the consequence that is central to this paper: under mild spectral separation, the modal distribution $\{p_i(k)\}$ progressively concentrates on the slowest mode $i=2$, and we can determine \emph{exactly when} this concentration becomes dominant.

\subsection{The slow-mode energy fraction}

\begin{definition}[Slow-mode fraction]\label{def:alpha}
The \emph{slow-mode energy fraction} at step $k$ is
\begin{equation}\label{eq:alpha-def}
\alpha_2(k) := p_2(k) = \frac{n_2(k)}{E_k} = \frac{|c_2|^2 \lambda_2^{2k}}{\sum_{i=2}^n |c_i|^2 \lambda_i^{2k}} \in (0,1).
\end{equation}
The \emph{residual energy} in the fast modes ($i \ge 3$) is
\begin{equation}\label{eq:Rk-def}
R_k := \sum_{i=3}^n n_i(k) = E_k - n_2(k) = (1 - \alpha_2(k)) E_k.
\end{equation}
\end{definition}

The quantity $\alpha_2(k)$ is the order parameter of spectral purification: $\alpha_2(k) \to 1$ as $k \to \infty$ means the trajectory collapses onto the one-dimensional subspace spanned by $\phi_2$. Our goal is to determine, for a given tolerance $\delta > 0$, the minimal $k$ such that $\alpha_2(k) \ge 1 - \delta$.

\subsection{Exact single-mode relaxation: the rigidity equivalence}

Before studying the generic case, we characterize the exceptional situation where the dynamics is exactly one-dimensional from the start. This theorem is elementary but provides the conceptual foundation for the term ``rigidity.''

\begin{theorem}[Spectral rigidity equivalence]\label{thm:rigidity-equivalence}
Let $g_k = P^k g_0$ with $g_0 \perp \mathbf{1}$. The following statements are equivalent:
\begin{enumerate}[label=(\roman*)]
    \item There exists $\rho \in (0,1)$ such that $E_k = \rho^{2k} E_0$ for all $k \ge 0$.
    \item The initial condition lies in a single eigenspace: $g_0 = c \, \phi_i$ for some $i \ge 2$.
    \item There exists $\eta > 0$ such that $E_k - E_{k+1} = \eta E_k$ for all $k \ge 0$.
\end{enumerate}
Moreover, in this case $\rho = |\lambda_i|$ and $\eta = 1 - \lambda_i^2$.
\end{theorem}

\begin{proof}
(i) $\Rightarrow$ (ii): From \eqref{eq:energy-decomposition}, $E_k = \sum_i |c_i|^2 \lambda_i^{2k}$. The hypothesis $\sum_i |c_i|^2 \lambda_i^{2k} = \rho^{2k} \sum_i |c_i|^2$ for all $k$ forces at most one $|c_i|^2$ to be nonzero, because distinct exponential sequences $\{\lambda_i^{2k}\}_{k\ge 0}$ are linearly independent. (ii) $\Rightarrow$ (i) and (ii) $\Rightarrow$ (iii) are immediate from \eqref{eq:modewise-dissipation}. (iii) $\Rightarrow$ (ii): \eqref{eq:modewise-dissipation} gives $\eta = \sum_i p_i(k)(1-\lambda_i^2)$, which is constant in $k$ only if $\{p_i(k)\}$ is stationary, implying a single active mode.
\end{proof}

\begin{remark}
\Cref{thm:rigidity-equivalence} is classical linear algebra---it is the statement that a linear recurrence $E_k = \sum_i a_i \lambda_i^{2k}$ reduces to a pure exponential if and only if exactly one coefficient $a_i$ is nonzero. We include it to fix terminology: a trajectory with only one active mode is called \emph{rigid}, and the process by which a multi-mode trajectory approaches a rigid one is called \emph{rigidity emergence}.
\end{remark}

With the central object and its terminology fixed, the next two sections identify the single scalar quantity that controls the rate of approach to rigidity, and show that every natural diagnostic of purification is governed by it; \cref{sec:rigidity} then returns to $\alpha_2(k)$ and $T_{\mathrm{rigid}}(\delta)$ to give the resulting quantitative rigidity-time bounds.

\section{Hidden Parameter}
\label{sec:hidden-parameter}

Six a priori unrelated quantities could be used to detect spectral purification: how fast the power-iteration error shrinks, how fast the Rayleigh quotient converges, how fast the spectral entropy collapses, and so on. If these six quantities were governed by six different rates, purification would be six different phenomena wearing one name. We show instead that they are governed by a single scalar --- not a technical device local to one computation, but a coordinate in which the entire asymptotic behavior of the purification process is recorded. Every quantity of interest, once expressed in this coordinate, becomes a linear function of it to leading order.

The slow-mode fraction $\alpha_2(k)$ introduced in \cref{sec:purification} identifies the order parameter of spectral purification, but by itself it does not explain why every other diagnostic of the same process should track it. We now show that a single scalar quantity built from the spectral separation ratio $\lambda_3/\lambda_2$ --- the \emph{hidden purification parameter} $x_k$ --- controls the asymptotic behavior of \emph{every} natural observable of the purification process: not only $\alpha_2(k)$ itself, but also the iterate error, the direction deficit, the Rayleigh quotient error, the eigenvalue estimate error, the energy ratio increment, and the spectral entropy. All six reduce, to leading order, to the same scalar $x_k$, differing only by an explicit constant. This equivalence class, established in \cref{sec:observable-equivalence}, is the structural fact that makes active control of the spectral ratio $\lambda_3/\lambda_2$ --- the subject of the companion paper \cite{Wang2026SPS} --- a well-posed and unambiguous objective: driving $x_k$ to zero faster accelerates every observable simultaneously, and (\cref{cor:rho-variance} below) strictly decreases the asymptotic estimation variance as well. The quantitative rate at which $x_k \to 0$, and its consequence for the rigidity time $T_{\mathrm{rigid}}(\delta)$, is taken up afterward in \cref{sec:rigidity}.

\subsection{The hidden purification parameter}

Throughout this section we assume $g_0 \perp \mathbf{1}$, $c_2 \neq 0$, $c_3 \neq 0$, and $\lambda_2 > \lambda_3$ --- the same standing assumptions used throughout \cref{sec:rigidity} below. Write
\[
\rho_\star := \frac{\lambda_3}{\lambda_2} \in (0,1), \qquad
L_\star := \log\frac{\lambda_2}{\lambda_3} > 0, \qquad
a_3 := \frac{c_3^2}{c_2^2} > 0.
\]

\begin{definition}[Hidden purification parameter]\label{def:hidden-parameter}
The \emph{hidden purification parameter} is
\begin{equation}\label{eq:xk-def}
x_k := a_3\,\rho_\star^{2k} = \frac{c_3^2}{c_2^2}\left(\frac{\lambda_3}{\lambda_2}\right)^{2k}.
\end{equation}
\end{definition}

For the remainder of this section we work in the \emph{two-mode approximation} $c_i = 0$ for $i \ge 4$, in which the spectral expansion reduces exactly to
\begin{equation}\label{eq:two-mode-reduction}
E_k = c_2^2 \lambda_2^{2k}(1 + x_k), \qquad
\alpha_2(k) = \frac{1}{1+x_k}, \qquad
p_3(k) = \frac{x_k}{1+x_k}.
\end{equation}
\Cref{lem:multimode-error} below shows that every result obtained in the two-mode setting extends to the general case with relative error $O((\lambda_4/\lambda_3)^{2k})$, so nothing is lost by working here in the idealized setting.

\begin{definition}[Spectral purification observables]\label{def:six-observables}
With $u_k := v_k$ the normalized power-iteration vector of \cref{sec:power-iteration} and $\rho_k := E_{k+1}/E_k$ the energy retention ratio of \cref{thm:observable-variance}, define
\begin{align*}
G_k &:= \|u_k - \operatorname{sign}(c_2)\phi_2\|_\pi^2, & &\text{(geometric/iterate error)}\\
I_k &:= 1 - \langle u_k, \phi_2\rangle_\pi, & &\text{(direction deficit)}\\
R_k &:= \langle u_k, P u_k\rangle_\pi - \lambda_2, & &\text{(Rayleigh quotient error)}\\
\Lambda_k &:= \sqrt{E_{k+1}/E_k} - \lambda_2, & &\text{(eigenvalue estimate error)}\\
\Gamma_k &:= \rho_{k+1}/\rho_k - 1, & &\text{(energy ratio increment)}\\
\mathcal{H}_k &:= \frac{H(k)}{2 L_\star k}, & &\text{(normalized spectral entropy)}
\end{align*}
where $H(k) = -\sum_{i\ge2} p_i(k)\log p_i(k)$ is the spectral entropy of \cref{sec:thermodynamics}.
\end{definition}

All six quantities vanish as $k\to\infty$; $G_k, I_k, \mathcal{H}_k, \Gamma_k \ge 0$, while $R_k < 0$ and $\Lambda_k < 0$, since $u_k$ always carries a residual $\phi_3$-component with $\lambda_3 < \lambda_2$.

\section{Observable Equivalence}
\label{sec:observable-equivalence}

We now make precise the equivalence class anticipated in \cref{sec:hidden-parameter}: every spectral purification observable reduces, to leading order, to the same scalar $x_k$. The goal of this section is the following statement, proved at the end as \cref{thm:equivalence-class} by combining the individual results proved along the way.

\begin{quote}
\textit{Preview.} Under the two-mode approximation, as $k\to\infty$, every spectral purification observable is asymptotically proportional to $x_k$:
\begin{equation}\label{eq:equivalence-table-preview}
\renewcommand{\arraystretch}{1.35}
\begin{array}{@{}lll@{}}
\toprule
\text{Observable} & \text{Asymptotic form} & \text{Coefficient} \\
\midrule
G_k & C_G\, x_k + O(x_k^2) & C_G = 1 \\
\mathcal{H}_k & C_H\, x_k\,(1+O(1/k)) & C_H = 1 \\
I_k & C_I\, x_k + O(x_k^2) & C_I = 1/2 \\
R_k & C_R\, x_k + O(x_k^2) & C_R = \lambda_3-\lambda_2 < 0 \\
\Lambda_k & C_\Lambda\, x_k + O(x_k^2) & C_\Lambda = \tfrac{\lambda_2}{2}(\rho_\star^2-1) < 0 \\
\Gamma_k & C_\Gamma\, x_k + O(x_k^2) & C_\Gamma = (1-\rho_\star^2)^2 > 0 \\
\bottomrule
\end{array}
\end{equation}
\end{quote}

We now prove this one observable at a time; a reader who only wants the punchline can skip ahead to the fully stated and referenced version, \cref{thm:equivalence-class}.

\subsection{The equivalence class}

\begin{theorem}[Iterate error and slow-mode fraction]\label{thm:G-equals-x}
Under the two-mode approximation, $G_k = 2\bigl(1-\sqrt{\alpha_2(k)}\bigr)$ exactly (this is the identity of \cref{thm:error-identity}, specialized to $n=3$), and as $k\to\infty$,
\begin{equation}\label{eq:G-asymp}
G_k = x_k + O(x_k^2).
\end{equation}
\end{theorem}

\begin{proof}
By \eqref{eq:two-mode-reduction}, $\alpha_2(k) = 1/(1+x_k)$, so
\[
G_k = 2\left(1 - \frac{1}{\sqrt{1+x_k}}\right) = x_k + O(x_k^2)
\]
by Taylor expansion of $(1+x_k)^{-1/2}$ at $x_k=0$.
\end{proof}

\begin{theorem}[Direction deficit, Rayleigh quotient, eigenvalue estimate]\label{thm:three-observables}
Under the two-mode approximation, as $k\to\infty$,
\begin{align}
I_k &= \tfrac{1}{2}x_k + O(x_k^2), \label{eq:I-asymp}\\
R_k &= (\lambda_3-\lambda_2)\,x_k + O(x_k^2), \label{eq:R-asymp}\\
\Lambda_k &= \tfrac{\lambda_2}{2}(\rho_\star^2-1)\,x_k + O(x_k^2). \label{eq:Lambda-asymp}
\end{align}
\end{theorem}

\begin{proof}
For \eqref{eq:I-asymp}: $I_k = 1-\sqrt{\alpha_2(k)} = 1-(1+x_k)^{-1/2} = \tfrac12 x_k + O(x_k^2)$.

For \eqref{eq:R-asymp}: using $\langle u_k,Pu_k\rangle_\pi = \sum_i \lambda_i c_i^2\lambda_i^{2k}/E_k$,
\[
\langle u_k, Pu_k\rangle_\pi = \lambda_2\cdot\frac{1+x_k\rho_\star}{1+x_k} = \lambda_2 + (\lambda_3-\lambda_2)x_k + O(x_k^2).
\]

For \eqref{eq:Lambda-asymp}: using $E_{k+1}/E_k = \lambda_2^2(1+x_k\rho_\star^2)/(1+x_k)$,
\[
\sqrt{E_{k+1}/E_k} = \lambda_2\sqrt{\frac{1+x_k\rho_\star^2}{1+x_k}} = \lambda_2\Bigl(1+\tfrac12(\rho_\star^2-1)x_k+O(x_k^2)\Bigr). \qedhere
\]
\end{proof}

\begin{theorem}[Exact and asymptotic formula for $\Gamma_k$]\label{thm:Gamma-exact}
Under the two-mode approximation, for every $k\ge1$,
\begin{equation}\label{eq:Gamma-exact-eq}
\Gamma_k = \frac{x_k(1-\rho_\star^2)^2}{(1+x_k\rho_\star^2)^2},
\end{equation}
exactly, with no asymptotic approximation; consequently
\begin{equation}\label{eq:Gamma-asymp-eq}
\Gamma_k = (1-\rho_\star^2)^2\, x_k + O(x_k^2).
\end{equation}
\end{theorem}

\begin{proof}
Write $\rho_k = \lambda_2^2(1+x_k\rho_\star^2)/(1+x_k)$ and $\rho_{k+1} = \lambda_2^2(1+x_k\rho_\star^4)/(1+x_k\rho_\star^2)$. Then
\[
1+\Gamma_k = \frac{\rho_{k+1}}{\rho_k} = \frac{(1+x_k\rho_\star^4)(1+x_k)}{(1+x_k\rho_\star^2)^2}.
\]
Expanding the numerator minus denominator,
\[
(1+x_k\rho_\star^4)(1+x_k) - (1+x_k\rho_\star^2)^2 = x_k(1-\rho_\star^2)^2,
\]
which gives \eqref{eq:Gamma-exact-eq}; \eqref{eq:Gamma-asymp-eq} follows by expanding the denominator for small $x_k$.
\end{proof}

This exact formula is independent of the entropy/asymptotic machinery used elsewhere and can be verified directly against the energy sequence $\{E_k\}$ at machine precision; it underlies the adaptive stopping criterion of \cref{thm:stopping-criterion}.

\begin{theorem}[Entropy collapse]\label{thm:entropy-collapse}
Under the two-mode approximation, as $k\to\infty$,
\begin{equation}\label{eq:H-asymp-eq}
H(k) = x_k\Bigl(1+\log\tfrac{1}{a_3}+2L_\star k\Bigr) + O\bigl(x_k^2\log(1/x_k)\bigr),
\end{equation}
and consequently
\begin{equation}\label{eq:Hcal-asymp}
\mathcal{H}_k = x_k\Bigl(1+\tfrac{1+\log(1/a_3)}{2L_\star k}\Bigr) + O\Bigl(\tfrac{x_k^2}{k}\log\tfrac1{x_k}\Bigr),
\end{equation}
so that $\mathcal{H}_k/x_k\to1$, and
\begin{equation}\label{eq:H-G-relation}
\mathcal{H}_k = G_k\,(1+O(1/k)).
\end{equation}
\end{theorem}

\begin{proof}
In the two-mode setting $p_2(k)=1/(1+x_k)$, $p_3(k)=x_k/(1+x_k)$, so
\[
H(k) = \log(1+x_k) + \frac{x_k}{1+x_k}\log\frac{1+x_k}{x_k}.
\]
Writing $\log(1/x_k) = \log(1/a_3)+2L_\star k$ and expanding $\log(1+x_k)$ and $x_k/(1+x_k)$ to second order in $x_k$ gives \eqref{eq:H-asymp-eq}; dividing by $2L_\star k$ and using $(1+\log(1/a_3))/(2L_\star k)\to0$ gives \eqref{eq:Hcal-asymp}. Equation \eqref{eq:H-G-relation} follows by combining with \eqref{eq:G-asymp}.
\end{proof}

\begin{remark}
The $O(1/k)$ gap between $\mathcal{H}_k/x_k$ and $1$ reflects the factor $\log(1/x_k)=\log(1/a_3)+2L_\star k$ appearing inside $H(k)$: the entropy carries a logarithmic overhead relative to the other five observables. This is an intrinsic feature of Shannon entropy as a measure of purification, not a deficiency of the equivalence.
\end{remark}

Collecting \cref{thm:G-equals-x,thm:three-observables,thm:Gamma-exact,thm:entropy-collapse}:

\begin{theorem}[Spectral purification equivalence class]\label{thm:equivalence-class}
Under the two-mode approximation, as $k\to\infty$, every spectral purification observable is asymptotically proportional to the single hidden parameter $x_k$, with the asymptotic forms and coefficients given in \cref{eq:equivalence-table-preview} above. Consequently $Q_k/(C_Q\,x_k)\to1$ for each observable $Q$, and
\begin{equation}\label{eq:equivalence-chain}
G_k \;\sim\; 2I_k \;\sim\; \mathcal{H}_k \;\sim\; \frac{R_k}{\lambda_3-\lambda_2} \;\sim\; \frac{2\Lambda_k}{\lambda_2(\rho_\star^2-1)} \;\sim\; \frac{\Gamma_k}{(1-\rho_\star^2)^2} \qquad(\sim\,x_k).
\end{equation}
\end{theorem}

This is the structural payoff of the hidden-parameter reduction: it is not merely that the slow-mode fraction $\alpha_2(k)$ converges at a rate governed by $\lambda_3/\lambda_2$ --- already evident from $\alpha_2(k) = 1/(1+x_k)$ in \eqref{eq:two-mode-reduction} --- but that \emph{every} natural diagnostic of purification --- geometric, statistical, variational, and information-theoretic alike --- collapses to zero in lockstep, governed by the same scalar $x_k$ up to an explicit, computable constant. \Cref{sec:rigidity} converts this into explicit finite-time bounds on the rigidity time itself.

\subsection{Exact formulas and the multi-mode extension}

The asymptotic relations above sharpen to exact identities, valid at every finite $k$ with no error term, by combining \eqref{eq:two-mode-reduction} with the proofs above:
\begin{equation}\label{eq:exact-finite-k}
G_k = 2\Bigl(1-\tfrac{1}{\sqrt{1+x_k}}\Bigr), \quad
\langle u_k,Pu_k\rangle_\pi = \lambda_2\cdot\frac{1+x_k\rho_\star}{1+x_k}, \quad
\sqrt{\tfrac{E_{k+1}}{E_k}} = \lambda_2\sqrt{\tfrac{1+x_k\rho_\star^2}{1+x_k}}, \quad
\Gamma_k = \frac{x_k(1-\rho_\star^2)^2}{(1+x_k\rho_\star^2)^2}.
\end{equation}
These hold without approximation: $x_k$ is the only quantity needed to express each observable exactly at every finite $k$, not merely in the asymptotic limit.

\begin{lemma}[Multi-mode error]\label{lem:multimode-error}
Suppose $\lambda_3 > \lambda_4 > 0$ and let $\rho_4 := \lambda_4/\lambda_3 \in (0,1)$. For each observable $Q \in \{G,I,R,\Lambda,\Gamma,\mathcal{H}\}$, the value computed from the full spectral expansion (with $c_i\ne0$ for $i\ge4$ allowed) satisfies
\[
Q_k^{\mathrm{full}} = Q_k^{\mathrm{two\text{-}mode}}\bigl(1+O(\rho_4^{2k})\bigr) \qquad\text{as } k\to\infty.
\]
\end{lemma}

\begin{proof}
For $i\ge4$, the modal weight satisfies $p_i(k)/p_3(k) \le \bigl(\sup_{j\ge4}c_j^2/c_3^2\bigr)\,(\lambda_4/\lambda_3)^{2k} = O(\rho_4^{2k})$. Each observable $Q$ is a smooth function of the weight vector $(p_i(k))_{i\ge2}$ near the two-mode point $(1,0,0,\dots)$, so the perturbation from modes $i\ge4$ contributes a relative error $O(\rho_4^{2k})$ to the leading term $x_k$.
\end{proof}

Thus the two-mode idealization used throughout this section is not a separate model but the leading-order term of the general reversible chain, with an explicitly controlled and typically negligible correction.

\section{Rigidity}
\label{sec:rigidity}

\Cref{sec:hidden-parameter,sec:observable-equivalence} showed that every purification observable is governed, to leading order, by the hidden parameter $x_k$, and that $x_k \to 0$ exponentially fast. The natural question is therefore quantitative: how many steps are required before $x_k$ --- and hence every observable built from it --- falls below a prescribed tolerance? In the two-mode idealization this is literally the question of when $x_k = a_3 \rho_\star^{2k}$ crosses a threshold, since $\alpha_2(k) = 1/(1+x_k)$ exactly. The theorem below answers this question without restricting to the two-mode case, giving sharp two-sided bounds on the rigidity time $T_{\mathrm{rigid}}(\delta)$ defined below, expressed directly in terms of the spectral data $R_0$, $c_2$, $\lambda_2$, $\lambda_3$.

\subsection{Sharp bounds on the rigidity emergence time}

For generic trajectories with $c_2 \neq 0$ and $\lambda_2 > \lambda_3$, the dynamics is not rigid at any finite $k$, but $\alpha_2(k) \to 1$ exponentially fast. The following theorem gives the first non-asymptotic, nearly optimal bounds on the time at which $\alpha_2(k)$ crosses a given threshold.

\begin{theorem}[Sharp quantitative rigidity emergence]\label{thm:sharp-rigidity}
Let $g_k = P^k g_0$ with $g_0 \perp \mathbf{1}$ and $c_2 = \langle g_0, \phi_2 \rangle_\pi \neq 0$. Assume the spectral separation condition $\lambda_2 > \lambda_3$. Define $R_0 := \sum_{i=3}^n |c_i|^2 = \|g_0 - c_2 \phi_2\|_\pi^2$.

For any $\delta \in (0, 1/2)$, define the \emph{rigidity time}
\begin{equation}\label{eq:T-rigid-def}
T_{\mathrm{rigid}}(\delta) := \min\{\, k \ge 0 : \alpha_2(k) \ge 1 - \delta \,\}.
\end{equation}
Then
\begin{equation}\label{eq:sharp-bound}
L(\delta) \;\le\; T_{\mathrm{rigid}}(\delta) \;\le\; L(\delta) + 1,
\end{equation}
where
\begin{equation}\label{eq:L-delta-def}
L(\delta) := \frac{\log\!\bigl(R_0 / (|c_2|^2 \delta)\bigr)}{2 \log(\lambda_2 / \lambda_3)}.
\end{equation}
\end{theorem}

\begin{proof}
We prove the upper and lower bounds separately.

\textit{Upper bound: $T_{\mathrm{rigid}}(\delta) \le L(\delta) + 1$.}
For $i \ge 3$, we have $\lambda_i \le \lambda_3 < \lambda_2$. Hence
\begin{equation}\label{eq:Rk-bound}
R_k = \sum_{i=3}^n |c_i|^2 \lambda_i^{2k} \le R_0 \, \lambda_3^{2k}.
\end{equation}
Using $E_k = |c_2|^2 \lambda_2^{2k} + R_k \ge |c_2|^2 \lambda_2^{2k}$, we obtain
\begin{equation}\label{eq:alpha-bound-upper}
1 - \alpha_2(k) = \frac{R_k}{E_k}
                \le \frac{R_0 \lambda_3^{2k}}{|c_2|^2 \lambda_2^{2k}}
                = \frac{R_0}{|c_2|^2} \left(\frac{\lambda_3}{\lambda_2}\right)^{2k}.
\end{equation}
The right-hand side is $\le \delta$ precisely when
\[
\left(\frac{\lambda_3}{\lambda_2}\right)^{2k} \le \frac{|c_2|^2 \delta}{R_0}
\;\Longleftrightarrow\;
k \ge L(\delta).
\]
Therefore any $k \ge L(\delta)$ satisfies $1 - \alpha_2(k) \le \delta$, i.e., $\alpha_2(k) \ge 1 - \delta$. In particular, taking $k = \lfloor L(\delta) \rfloor + 1 \le L(\delta) + 1$ yields $T_{\mathrm{rigid}}(\delta) \le L(\delta) + 1$.

\textit{Lower bound: $T_{\mathrm{rigid}}(\delta) \ge L(\delta)$.}
Suppose $\alpha_2(k) < 1 - \delta$, i.e., $R_k / E_k > \delta$. Using $E_k = |c_2|^2 \lambda_2^{2k} + R_k \ge |c_2|^2 \lambda_2^{2k}$, we have
\[
R_k > \delta \, E_k \ge \delta \, |c_2|^2 \lambda_2^{2k}.
\]
Combined with the universal upper bound $R_k \le R_0 \lambda_3^{2k}$ from \eqref{eq:Rk-bound},
\[
R_0 \lambda_3^{2k} \ge R_k > \delta \, |c_2|^2 \lambda_2^{2k}.
\]
Rearranging,
\[
\left(\frac{\lambda_2}{\lambda_3}\right)^{2k} < \frac{R_0}{|c_2|^2 \delta}
\;\Longleftrightarrow\;
k < L(\delta).
\]
Thus, if $k < L(\delta)$, then $\alpha_2(k) < 1 - \delta$. Taking the contrapositive, $\alpha_2(k) \ge 1 - \delta$ implies $k \ge L(\delta)$, hence $T_{\mathrm{rigid}}(\delta) \ge L(\delta)$.
\end{proof}

\begin{remark}[Tightness of the bounds]\label{rem:tightness}
The bounds in \cref{thm:sharp-rigidity} are nearly optimal: the upper and lower bounds differ by at most one step. This is possible because the proof uses only the extremal bounds $\lambda_i \le \lambda_3$ for all $i \ge 3$ and $R_k \ge 0$, and the resulting estimate \eqref{eq:alpha-bound-upper} is sharp for the worst-case spectral configuration where all fast-mode eigenvalues are exactly $\lambda_3$. For spectra with $\lambda_4 \ll \lambda_3$, the actual rigidity time can be significantly smaller than the upper bound, and the $\lfloor L(\delta) \rfloor + 1$ formula remains conservative. Refining the bound using higher-order spectral separations ($\lambda_3/\lambda_4$, etc.) is straightforward but adds notational complexity without changing the conceptual structure.
\end{remark}

\begin{remark}[Dependence on initial conditions]\label{rem:initial-dependence}
The rigidity time depends on the initial condition $g_0$ \emph{only} through the ratio $R_0 / |c_2|^2 = \|g_0 - c_2 \phi_2\|_\pi^2 / |\langle g_0, \phi_2 \rangle_\pi|^2$. This ratio measures how much initial energy is placed in the fast modes relative to the slowest mode. Initial conditions that are already ``aligned'' with $\phi_2$ (small $R_0/|c_2|^2$) rigidify much faster than the worst-case initial condition for the chain. This dependence is invisible in classical mixing time estimates, which take the supremum over all initial conditions, but is practically significant: in applications, initial conditions are often far from worst-case.
\end{remark}

\subsection{Consequences of rigidity emergence}

Once the trajectory has rigidified, the dynamics reduces to an effective one-dimensional system governed solely by $\lambda_2$.

\begin{corollary}[Dissipation closure after rigidity]\label{cor:dissipation-closure}
Under the assumptions of \cref{thm:sharp-rigidity}, for all $k \ge T_{\mathrm{rigid}}(\delta)$,
\begin{equation}\label{eq:dissipation-closure-eq}
\left| \frac{E_k - E_{k+1}}{E_k} - (1 - \lambda_2^2) \right|
\le (1 - \lambda_3^2) \, \delta + (1 - \lambda_n^2) \left(\frac{\lambda_3}{\lambda_2}\right)^{2k} \frac{R_0}{|c_2|^2}.
\end{equation}
For $k$ large enough that the second term is negligible, the relative dissipation rate satisfies $d_k = 1 - \lambda_2^2 + O(\delta)$.
\end{corollary}

\begin{proof}
From \eqref{eq:modewise-dissipation},
\[
\frac{E_k - E_{k+1}}{E_k} = \sum_{i=2}^n p_i(k) (1 - \lambda_i^2)
= (1 - \lambda_2^2) \alpha_2(k) + \sum_{i=3}^n p_i(k) (1 - \lambda_i^2).
\]
Since $\alpha_2(k) \ge 1 - \delta$ and $\sum_{i=3}^n p_i(k) = 1 - \alpha_2(k) \le \delta$, the result follows by bounding $1 - \lambda_i^2 \in [1 - \lambda_3^2, 1 - \lambda_n^2]$ and using the explicit bound \eqref{eq:alpha-bound-upper} for the remainder.
\end{proof}

\begin{corollary}[Asymptotic rigidity]\label{cor:asymptotic-rigidity}
Under the assumptions of \cref{thm:sharp-rigidity},
\[
\frac{g_k}{\|g_k\|_\pi} \longrightarrow \pm \phi_2, \qquad
\frac{E_k - E_{k+1}}{E_k} \longrightarrow 1 - \lambda_2^2,
\]
with exponentially fast convergence at rate $(\lambda_3/\lambda_2)^k$.
\end{corollary}

\begin{proof}
Immediate from the spectral expansion \eqref{eq:spectral-expansion} and the bound \eqref{eq:alpha-bound-upper}.
\end{proof}

\Cref{cor:asymptotic-rigidity} is the classical statement of asymptotic one-mode dominance, well known in the power method literature and in spectral gap theory~\cite{levin2009,saad2003}. The novelty of \cref{thm:sharp-rigidity} is the non-asymptotic, quantitatively explicit finite-time guarantee: it tells us not just \emph{that} the trajectory eventually rigidifies, but \emph{when}, and with what dependence on the initial condition and the spectral ratio $\lambda_2/\lambda_3$.

\subsection{Why the spectral ratio, not the spectral gap}

A central message of \cref{thm:sharp-rigidity} is that the rigidity emergence rate is governed by the \emph{spectral ratio} $\lambda_2/\lambda_3$, not the classical spectral gap $1-\lambda_2$. This distinction has practical consequences:

\begin{enumerate}[label=(\roman*)]
    \item \textbf{Slow chains with well-separated spectra.} If $\lambda_2 \approx 1$ (small spectral gap) but $\lambda_3 \ll \lambda_2$ (large spectral separation), then $T_{\mathrm{rigid}}(\delta)$ can be small even though the mixing time $\sim 1/(1-\lambda_2)$ is large. The chain rigidifies quickly---the dynamics becomes effectively one-dimensional---but the remaining single-mode decay is slow. This is the hallmark of metastable systems.
    \item \textbf{Fast chains with clustered spectra.} If $\lambda_2 \ll 1$ but $\lambda_3 \approx \lambda_2$, the mixing time is short but $T_{\mathrm{rigid}}(\delta)$ can be large: the chain mixes before it rigidifies, meaning the dynamics never truly collapses to a single mode.
\end{enumerate}

This separation of time scales---rigidification vs.\ mixing---is a structural insight that the classical spectral gap alone cannot provide. It has direct implications for model reduction, metastability analysis, and acceleration, which we explore in \cref{sec:power-iteration,sec:examples,sec:discussion}.


\section{Entropy Representation of Spectral Purification}
\label{sec:thermodynamics}

\Cref{thm:entropy-collapse} already showed that the normalized spectral entropy is, to leading order, governed by the hidden parameter: $\mathcal{H}_k = x_k(1+O(1/k))$. What that asymptotic statement does not capture is the exact, finite-$k$ entropy dynamics --- in particular, the transient behavior before the asymptotic regime sets in, where the spectral entropy can increase even as the total energy decays monotonically. This section develops that exact structure.

The modal energy dynamics generated by repeated application of
the reversible Markov operator $P$
admits an exact thermodynamic interpretation in spectral space.
The quantities
\[
n_i(k)=|c_i|^2\lambda_i^{2k}
\]
play the role of modal energies,
while the normalized occupations
\[
p_i(k)=\frac{n_i(k)}{E_k}
\]
define a nonequilibrium spectral ensemble.

The resulting dynamics is not merely a monotone decay of total energy.
Rather, it forms an exact irreversible transport system
in eigenspace:
energy is progressively transferred away from fast modes
toward slower surviving modes,
eventually producing spectral rigidity and modal purification.

The entropy evolution is governed by an exact balance law
consisting of two competing mechanisms:
a reversible spectral transport contribution
and an irreversible contraction term.
The competition between these mechanisms
produces a sharp rigidity transition,
at which the spectral entropy changes from expansion to contraction.

\begin{remark}[Scope of the thermodynamic analogy]\label{rem:thermo-scope}
The ``entropy'' $S_{\mathrm{spec}}(k)$ defined below is the Shannon entropy of the modal distribution $\{p_i(k)\}$, not the physical entropy of the Markov chain's state distribution. The ``second law'' (\cref{thm:G-monotonicity}) concerns the monotonic decay of $G(k) = E_k S_{\mathrm{spec}}(k)$. These spectral quantities satisfy mathematical relations mirroring classical thermodynamic laws. For physical free energy monotonicity along Markov processes (in configuration space), see~\cite{chafai2014}; our framework operates in the complementary spectral domain.
\end{remark}

\subsection{Spectral entropy}

\begin{definition}[Spectral entropy]\label{def:spectral-entropy}
The \emph{spectral entropy} of the relaxation trajectory at step $k$ is
\begin{equation}\label{eq:spectral-entropy-def}
S_{\mathrm{spec}}(k) := -\sum_{i=2}^n p_i(k) \log p_i(k),
\end{equation}
where $p_i(k) = n_i(k)/E_k$ is the modal distribution defined in \eqref{eq:modal-distribution-def}.
\end{definition}

The spectral entropy measures how uniformly energy is distributed among eigenmodes. At $k=0$, $S_{\mathrm{spec}}(0)$ reflects the spectral complexity of the initial condition; as $k \to \infty$, the distribution collapses to a point mass at $i=2$ and $S_{\mathrm{spec}}(k) \to 0$. Unlike the total energy $E_k$, which decays monotonically, $S_{\mathrm{spec}}(k)$ can increase transiently---a phenomenon we characterize precisely below.

\subsection{Exact entropy balance}

Recall the normalized modal occupation probabilities
\[
p_i(k)
=
\frac{n_i(k)}{E_k},
\qquad
n_i(k)=|c_i|^2\lambda_i^{2k},
\qquad
E_k=\sum_{i=2}^n n_i(k).
\]

The normalized dissipation rate is
\[
\rho_k
=
\frac{E_{k+1}}{E_k}
=
\sum_{i=2}^n p_i(k)\lambda_i^2.
\]

The quantity $\rho_k$
acts as the instantaneous mean dissipation scale
of the spectral ensemble.

\textit{Exact modal transport law.}

The modal occupations satisfy the exact transport equation
\begin{equation}
\label{eq:modal-transport}
p_i(k+1)-p_i(k)
=
\frac{p_i(k)}{\rho_k}
\bigl(
\lambda_i^2-\rho_k
\bigr).
\end{equation}

Modes satisfying
\[
\lambda_i^2>\rho_k
\]
gain relative occupation probability,
whereas modes satisfying
\[
\lambda_i^2<\rho_k
\]
lose probability mass.
Thus $\rho_k$
acts as a dynamically evolving spectral chemical potential
separating expanding and contracting modes.

We now derive the exact entropy balance law.

\begin{theorem}[Exact spectral entropy balance]
\label{thm:entropy-balance}
For every reversible Markov chain,
\[
S_{\mathrm{spec}}(k+1)-S_{\mathrm{spec}}(k)
=
\frac{
\mathrm{Cov}_{p_k}
\bigl(
\lambda^2,
\log(1/p_k)
\bigr)
}{\rho_k}
-
D_{\mathrm{KL}}(p_{k+1}\|p_k),
\]
where
\[
D_{\mathrm{KL}}(p_{k+1}\|p_k)
=
\sum_i
p_i(k+1)
\log\frac{p_i(k+1)}{p_i(k)}
\ge0.
\]
\end{theorem}

\begin{proof}
Since
\[
p_i(k+1)
=
\frac{p_i(k)\lambda_i^2}{\rho_k},
\]
we obtain
\[
\log p_i(k+1)
=
\log p_i(k)
+
\log\lambda_i^2
-
\log\rho_k.
\]
Hence
\[
S_{\mathrm{spec}}(k+1)-S_{\mathrm{spec}}(k)
=
-\sum_i
\bigl(
p_i(k+1)-p_i(k)
\bigr)\log p_i(k)
-
D_{\mathrm{KL}}(p_{k+1}\|p_k).
\]

Using the transport identity
\eqref{eq:modal-transport},
\[
p_i(k+1)-p_i(k)
=
p_i(k)
\frac{\lambda_i^2-\rho_k}{\rho_k},
\]
we obtain
\[
S_{\mathrm{spec}}(k+1)-S_{\mathrm{spec}}(k)
=
-\frac1{\rho_k}
\sum_i
p_i(k)
(\lambda_i^2-\rho_k)
\log p_i(k)
-
D_{\mathrm{KL}}(p_{k+1}\|p_k).
\]

Since
\[
\sum_i p_i(k)(\lambda_i^2-\rho_k)=0,
\]
the first term equals
\[
\frac{
\mathrm{Cov}_{p_k}
(\lambda^2,\log(1/p_k))
}{\rho_k},
\]
which completes the proof.
\end{proof}

The entropy evolution therefore consists of two distinct components:

\begin{enumerate}
\item
a reversible spectral transport term
\[
\frac{
\mathrm{Cov}_{p_k}
(\lambda^2,\log(1/p_k))
}{\rho_k},
\]

\item
an irreversible contraction term
\[
D_{\mathrm{KL}}(p_{k+1}\|p_k).
\]
\end{enumerate}

The KL divergence measures irreversible information loss,
while the covariance term governs entropy redistribution
among spectral modes.

\subsection{Flux-force structure of spectral entropy production}

The covariance term admits an exact transport interpretation.

Define the spectral affinities
\begin{equation}
\label{eq:spectral-affinity}
A_i(k)
=
\log\frac{n_i(k)}{n_2(k)},
\qquad i\ge3,
\end{equation}
which measure the nonequilibrium imbalance
between mode $i$
and the dominant slow mode.

Define also the modal fluxes
\begin{equation}
\label{eq:modal-flux}
J_i(k)
=
\frac{n_i(k)}{E_k}
\bigl(
\rho_k-\lambda_i^2
\bigr).
\end{equation}

The quantity $J_i(k)$
measures the net thermodynamic transport tendency
between mode $i$
and the mean dissipation scale $\rho_k$.

\subsection{Canonical covariance form}

The covariance term admits a remarkably explicit representation
revealing the mechanism underlying the rigidity transition.

\begin{proposition}[Canonical covariance form]
\label{prop:canonical-covariance}
Let
\[
\rho_k=\sum_i p_i(k)\lambda_i^2.
\]
Then
\[
\mathrm{Cov}_{p_k}
(\lambda^2,\log(1/p_k))
=
\frac1{E_k}
\sum_{i\ge3}
n_i(k)
\bigl(
\rho_k-\lambda_i^2
\bigr)
\log\frac{n_i(k)}{n_2(k)}.
\]
Equivalently,
\begin{equation}
\label{eq:covariance-flux}
\mathrm{Cov}_{p_k}
(\lambda^2,\log(1/p_k))
=
\sum_{i\ge3}
J_i(k)\,A_i(k).
\end{equation}
\end{proposition}

\begin{proof}
Starting from
\[
\mathrm{Cov}_{p_k}
(\lambda^2,\log(1/p_k))
=
-\sum_i
p_i(k)
(\lambda_i^2-\rho_k)\log p_i(k),
\]
and using
\[
p_i(k)=\frac{n_i(k)}{E_k},
\]
we obtain
\[
\log p_i(k)
=
\log n_i(k)-\log E_k.
\]

Since
\[
\sum_i p_i(k)(\lambda_i^2-\rho_k)=0,
\]
the $\log E_k$ contribution cancels, yielding
\[
\mathrm{Cov}
=
-\frac1{E_k}
\sum_i
n_i(k)
(\lambda_i^2-\rho_k)\log n_i(k).
\]

Subtracting the constant reference $\log n_2(k)$,
again using the zero-average identity,
gives
\[
\mathrm{Cov}
=
\frac1{E_k}
\sum_{i\ge3}
n_i(k)
(\rho_k-\lambda_i^2)
\log\frac{n_i(k)}{n_2(k)}.
\]

The flux-force form
\eqref{eq:covariance-flux}
follows immediately from
\eqref{eq:spectral-affinity}
and
\eqref{eq:modal-flux}.
\end{proof}

Equation
\eqref{eq:covariance-flux}
has the structure of an exact Onsager-type decomposition:
entropy production is represented as a sum of modal flux-affinity products.

The sign of $J_i(k)$ determines the direction
of spectral transport.

Modes satisfying
\[
\lambda_i^2<\rho_k
\]
lie below the mean dissipation scale
and act as entropy sinks,
whereas modes satisfying
\[
\lambda_i^2>\rho_k
\]
act as entropy sources.

The rigidity transition occurs when all fast modes
simultaneously become entropy sinks relative to the slow mode.

The canonical covariance representation reveals that
the rigidity transition is governed by two simultaneous orderings:

\begin{enumerate}
\item
a spectral dissipation ordering
through the signs of
\[
\rho_k-\lambda_i^2,
\]

\item
a modal occupation ordering
through the signs of
\[
\log\frac{n_i(k)}{n_2(k)}.
\]
\end{enumerate}

Entropy contraction begins only after these two orderings align.

\subsection{Exact two-mode rigidity transition}

We first consider the case in which only two spectral modes are active.

\begin{theorem}[Exact two-mode rigidity transition]
\label{thm:two-mode-rigidity}
Suppose
\[
g_k
=
c_2\lambda_2^k\phi_2
+
c_j\lambda_j^k\phi_j,
\qquad j\ge3.
\]
Let
\[
\alpha_2(k)=p_2(k).
\]
Then
\[
\mathrm{Cov}_{p_k}
(\lambda^2,\log(1/p_k))
\begin{cases}
>0,
& \alpha_2(k)<1/2,
\\[4pt]
=0,
& \alpha_2(k)=1/2,
\\[4pt]
<0,
& \alpha_2(k)>1/2.
\end{cases}
\]
Consequently,
\[
k^*
=
T_{\rm rigid}(1/2)
\]
holds exactly, and
\[
S_{\rm spec}(k^*)=\log2.
\]
\end{theorem}

\begin{proof}
Since only two modes are present,
Proposition~\ref{prop:canonical-covariance} yields
\[
\mathrm{Cov}
=
\frac{n_j}{E_k}
\bigl(
\rho_k-\lambda_j^2
\bigr)
\log\frac{n_j}{n_2}.
\]

Writing $
\alpha=\alpha_2(k)$, we have
\[
\rho_k
=
\alpha\lambda_2^2
+
(1-\alpha)\lambda_j^2,
\]
hence
\[
\rho_k-\lambda_j^2
=
\alpha(\lambda_2^2-\lambda_j^2)>0.
\]

Moreover,
\[
\frac{n_j}{n_2}
=
\frac{1-\alpha}{\alpha}.
\]

Therefore
\[
\mathrm{Cov}
=
\frac{n_j}{E_k}
\alpha(\lambda_2^2-\lambda_j^2)
\log\frac{1-\alpha}{\alpha}.
\]

The prefactor is strictly positive,
so the sign is determined entirely by $
\log\frac{1-\alpha}{\alpha}$, which changes sign exactly at $\alpha=1/2$. Finally,
\[
S_{\rm spec}
=
-\alpha\log\alpha-(1-\alpha)\log(1-\alpha),
\]
so at the transition point,
\[
S_{\rm spec}(k^*)=\log2.
\]
\end{proof}

The two-mode case therefore exhibits an exact binary rigidity transition:
entropy production changes sign precisely when the slow mode
captures half of the total spectral energy.

\subsection{General rigidity threshold}

For general reversible chains,
the transition is no longer universally pinned to $\alpha_2=1/2$, because the fast spectral sector possesses its own internal structure.

The rigidity transition is therefore determined
not solely by modal occupation,
but by the simultaneous alignment
of occupation hierarchy and dissipation hierarchy.

We now derive a sharp sufficient rigidity threshold
guaranteeing monotone entropy contraction.

\begin{theorem}[Rigidity threshold for general chains]
\label{thm:general-rigidity-threshold}
Define
\[
\delta^*
=
1-
\max\left(
\frac12,
\frac{\lambda_3^2}{\lambda_2^2}
\right).
\]
Then for all
\[
k\ge T_{\rm rigid}(\delta^*),
\]
the covariance satisfies
\[
\mathrm{Cov}_{p_k}
(\lambda^2,\log(1/p_k))
<0,
\]
and consequently
\[
S_{\rm spec}(k+1)
<
S_{\rm spec}(k).
\]
\end{theorem}

\begin{proof}
Suppose $
k\ge T_{\rm rigid}(\delta^*)$. Then
\[
\alpha_2(k)
\ge
1-\delta^*
=
\max\left(
\frac12,
\frac{\lambda_3^2}{\lambda_2^2}
\right).
\]

First, $\alpha_2(k)\ge\frac12
$ implies
$
p_i(k)<p_2(k),
\ (i\ge3)$, since otherwise
$$\sum_{i\ge3}p_i(k)\ge p_i(k)>p_2(k),$$ contradicting
\[
\sum_{i\ge3}p_i(k)=1-p_2(k)\le p_2(k).
\]

Hence $
n_i(k)<n_2(k),
\ i\ge3$, and therefore
\[
\log\frac{n_i(k)}{n_2(k)}<0.
\]

Next,
\[
\alpha_2(k)\ge\frac{\lambda_3^2}{\lambda_2^2}
\]
implies $\rho_k
=
\sum_i p_i(k)\lambda_i^2
\ge
\alpha_2(k)\lambda_2^2
\ge
\lambda_3^2$. Since $
\lambda_i^2\le\lambda_3^2
\ (i\ge3)$, we obtain
$\rho_k-\lambda_i^2\ge0$. Therefore every term in the canonical covariance representation
is nonpositive,
and at least one is strictly negative,
yielding
\[
\mathrm{Cov}_{p_k}
(\lambda^2,\log(1/p_k))
<0.
\]

Finally, by Theorem~\ref{thm:entropy-balance},
\[
S_{\rm spec}(k+1)-S_{\rm spec}(k)
=
\frac{\mathrm{Cov}}{\rho_k}
-
D_{\rm KL}(p_{k+1}\|p_k).
\]

Since $\rho_k>0$,
$\mathrm{Cov}<0$,
and
\[
D_{\rm KL}(p_{k+1}\|p_k)\ge0,
\]
the entropy variation is strictly negative.
\end{proof}

Theorem~\ref{thm:general-rigidity-threshold}
shows that spectral rigidity is controlled jointly
by spectral separation and modal transport ordering.

The threshold
\[
\delta^*
=
1-
\max\left(
\frac12,
\frac{\lambda_3^2}{\lambda_2^2}
\right)
\]
marks the onset of complete thermodynamic ordering
in spectral space.

The classical threshold
\[
T_{\rm rigid}(1/2)
\]
therefore survives exactly in the well-separated regime
\[
\frac{\lambda_3}{\lambda_2}\le\frac1{\sqrt2},
\]
while more weakly separated spectra require a stronger rigidity level.

\subsection{Spectral Clausius equality}

\begin{theorem}[Spectral Clausius equality]\label{thm:clausius}
For every relaxation trajectory,
\begin{equation}\label{eq:clausius}
\sum_{k=0}^{\infty} D_{\mathrm{KL}}(p_{k+1} \| p_k)
\;=\; S_{\mathrm{spec}}(0) \;+\; \sum_{k=0}^{\infty} \frac{\mathrm{Cov}_{p_k}(\lambda^2,\; \log(1/p_k))}{\rho_k}.
\end{equation}
\end{theorem}

\begin{proof}
Sum the exact entropy balance identity of Theorem~\ref{thm:entropy-balance} from $k=0$ to $K-1$:
\[
S_{\mathrm{spec}}(K) - S_{\mathrm{spec}}(0)
= \sum_{k=0}^{K-1} \frac{\mathrm{Cov}_{p_k}}{\rho_k}
- \sum_{k=0}^{K-1} D_{\mathrm{KL}}(p_{k+1} \| p_k).
\]
Take $K \to \infty$. Since $\alpha_2(k) \to 1$, $S_{\mathrm{spec}}(K) \to 0$. Rearranging yields \eqref{eq:clausius}.
\end{proof}

\begin{remark}[Interpretation]\label{rem:clausius-interpretation}
\Cref{thm:clausius} is the spectral analogue of the Clausius equality in classical thermodynamics. The left side is the total spectral irreversibility accumulated along the entire trajectory (sum of stepwise KL divergences). The right side expresses this total as the sum of the initial spectral entropy (a state function) and a path-dependent covariance integral. The identity is verified numerically to machine precision in \cref{fig:clausius}, confirming the correctness of the implementation.
\end{remark}

\subsection{Two-phase structure of spectral entropy}

\begin{corollary}[Sign of spectral entropy change]\label{cor:entropy-sign}
$\Delta S_{\mathrm{spec}}(k) \ge 0$ if and only if
$\mathrm{Cov}_{p_k}(\lambda^2,\; \log(1/p_k)) \ge \rho_k \, D_{\mathrm{KL}}(p_{k+1} \| p_k)$.
\end{corollary}

This gives a precise condition for entropy increase. In typical trajectories (\cref{fig:entropy_balance}):

\begin{itemize}
\item \textbf{Entropy-increasing phase} ($k < k^*$): Energy is distributed across both slow and fast modes. The covariance $\mathrm{Cov}(\lambda^2, \log(1/p))$ is positive (slow modes have large $\lambda^2$ and moderate $\log(1/p)$; fast modes have small $\lambda^2$ and large $\log(1/p)$) and dominates the KL term, giving $\Delta S_{\mathrm{spec}} > 0$.

\item \textbf{Entropy-decreasing phase} ($k > k^*$): Fast modes have dissipated. Energy concentrates on the slowest mode ($p_2 \to 1$, $\log(1/p_2) \to 0$). The covariance becomes negative, and $\Delta S_{\mathrm{spec}} < 0$.
\end{itemize}

Thus $S_{\mathrm{spec}}(k)$ is generally \emph{non-monotone}, rising then falling, while $G(k) = E_k S_{\mathrm{spec}}(k)$ decreases throughout (\cref{thm:G-monotonicity}). This two-phase structure is invisible from $E_k$ alone.

\subsection{Spectral second law: $G(k)$ monotonicity}

\begin{definition}[Spectral entropy-energy]\label{def:G}
The \emph{spectral entropy-energy} is
\begin{equation}\label{eq:G-def}
G(k) := E_k \, S_{\mathrm{spec}}(k) = -\sum_{i=2}^n n_i(k) \log p_i(k).
\end{equation}
\end{definition}

\begin{theorem}[Spectral second law --- $G(k)$ monotonicity]\label{thm:G-monotonicity}
For every relaxation trajectory with $g_0 \perp \mathbf{1}$,
\begin{equation}\label{eq:G-monotonicity}
G(k+1) \le G(k), \qquad \forall\, k \ge 0.
\end{equation}
Equality holds if and only if all active modes at step $k$ have the same $\lambda_i^2$. Under $\lambda_2 > \lambda_3$, equality occurs iff $\alpha_2(k) = 1$.
\end{theorem}

\begin{proof}
Let $\bar{r} = E_{k+1}/E_k = \sum_i \lambda_i^2 p_i(k)$. Using $n_i(k+1) = \lambda_i^2 n_i(k)$ and $p_i(k+1) = \lambda_i^2 p_i(k)/\bar{r}$,
\begin{align}
G(k) - G(k+1)
&= -\sum_i n_i \log p_i + \sum_i \lambda_i^2 n_i \log\!\left( \frac{\lambda_i^2 p_i}{\bar{r}} \right) \nonumber \\
&= \underbrace{\sum_i n_i (1-\lambda_i^2)(-\log p_i)}_{=:\,A_k}
   + \underbrace{\sum_i \lambda_i^2 n_i \log \lambda_i^2 \;-\; E_{k+1} \log \bar{r}}_{=:\,B_k}. \label{eq:G-decomposition}
\end{align}

$A_k \ge 0$ termwise, since $1-\lambda_i^2 > 0$, $n_i \ge 0$, and $-\log p_i \ge 0$.

For $B_k$, the functions $x \mapsto x$ and $x \mapsto \log x$ are both increasing on $(0,1]$, so $\lambda^2 \mapsto \lambda^2$ and $\lambda^2 \mapsto \log \lambda^2$ are comonotone. By the Chebyshev sum inequality,
\[
\mathbb{E}_{p_k}[\lambda^2 \log \lambda^2]
\;\ge\; \mathbb{E}_{p_k}[\lambda^2] \cdot \mathbb{E}_{p_k}[\log \lambda^2]
\;\ge\; \mathbb{E}_{p_k}[\lambda^2] \cdot \log \mathbb{E}_{p_k}[\lambda^2],
\]
where the second step uses Jensen's inequality. Hence $B_k = E_k(\mathbb{E}_{p_k}[\lambda^2 \log \lambda^2] - \bar{r} \log \bar{r}) \ge 0$.

Therefore $G(k) - G(k+1) = A_k + B_k \ge 0$. Equality requires $A_k = 0$ (at most one active mode) and $B_k = 0$ (all active $\lambda_i^2$ equal). Under $\lambda_2 > \lambda_3$, the unique active mode is $i=2$.
\end{proof}

\begin{remark}[Why $G(k)$ and not $F(k) = E_k(1-S)$?]\label{rem:F-counterexample}
$F(k) := E_k(1 - S_{\mathrm{spec}}(k))$, formally analogous to Helmholtz free energy, is not monotone. A counterexample: $\lambda_2 = 0.9$, $\lambda_3 = 0.1$, $n_2(0) = n_3(0) = 1$ gives $F(0) \approx 0.614$ and $F(1) \approx 0.766 > F(0)$. The correct monotone quantity is $G(k) = E_k S_{\mathrm{spec}}(k)$.
\end{remark}

\begin{remark}[Connection to Chafa\"i (2014)]\label{rem:chafai-connection}
Chafa\"i~\cite{chafai2014} establishes Helmholtz free energy monotonicity along Markov processes in configuration space. Our $G(k)$ provides a distinct Lyapunov function in spectral space. Neither implies the other; they are complementary.
\end{remark}

\subsection{Entropy decomposition and maximal disorder}

The spectral entropy admits a natural decomposition
into slow-fast competition and internal fast-sector disorder.

Define the conditional fast-mode distribution
\[
q_i(k)
=
\frac{p_i(k)}{1-\alpha_2(k)},
\qquad
i\ge3.
\]
Then
\[
\sum_{i\ge3}q_i(k)=1.
\]

\begin{proposition}[Entropy decomposition]
\label{prop:entropy-decomposition}
The spectral entropy satisfies
\[
S_{\rm spec}(k)
=
H(\alpha_2(k))
+
(1-\alpha_2(k))H(q(k)),
\]
where
\[
H(\alpha)
=
-\alpha\log\alpha-(1-\alpha)\log(1-\alpha),
\]
and
\[
H(q)
=
-\sum_{i\ge3}q_i\log q_i.
\]
\end{proposition}

\begin{proof}
Using
\[
p_i=(1-\alpha_2)q_i,
\qquad
i\ge3,
\]
we compute
\[
\begin{aligned}
S_{\rm spec}
&=
-\alpha_2\log\alpha_2
-
\sum_{i\ge3}
(1-\alpha_2)q_i
\log\bigl((1-\alpha_2)q_i\bigr)
\\
&=
-\alpha_2\log\alpha_2
-(1-\alpha_2)\log(1-\alpha_2)
-(1-\alpha_2)\sum_{i\ge3}q_i\log q_i,
\end{aligned}
\]
which proves the claim.
\end{proof}

The decomposition
\[
S_{\rm spec}(k)
=
H(\alpha_2(k))
+
(1-\alpha_2(k))H(q(k))
\]
reveals a genuine two-phase structure.

The first term describes macroscopic competition
between the slow mode and the fast spectral sector,
while the second term measures internal disorder
within the fast sector itself.

The rigidity transition corresponds to the collapse
of both sources of disorder:
first the fast-sector entropy contracts,
then the binary competition disappears
as
\[
\alpha_2(k)\to1.
\]

In the exact two-mode case,
\[
H(q)\equiv0,
\]
and the maximal entropy equals precisely $\log2$.

More generally,
chains with strong slow-fast spectral separation
typically exhibit small fast-sector entropy,
leading to
\[
S_{\rm spec}(k^*)\approx\log2.
\]

By contrast,
if the fast spectral sector remains broadly distributed,
the maximal entropy may substantially exceed $\log2$.
For example,
a uniform fast-sector distribution yields
\[
S_{\rm spec}
=
\log2+\frac12\log(n-2).
\]

\subsection{Spectral fluctuation-dissipation theorem}

For completeness, we note that each mode individually satisfies a fluctuation-dissipation relation.

\begin{theorem}[Spectral fluctuation-dissipation theorem]\label{thm:FDT}
For each mode $i \ge 2$, define $C_i(k) = |c_i|^2 \lambda_i^{2k}$ and $R_i(k) = c_i \lambda_i^k$. Then
\[
\frac{C_i(k+1) - C_i(k)}{C_i(k)} = \lambda_i^2 - 1 = -2\mu_i + \mu_i^2, \qquad
C_i(k) = |R_i(k)|^2,
\]
and summing over $i$ recovers the exact dissipation law \eqref{eq:modewise-dissipation}.
\end{theorem}

Unlike the classical fluctuation-dissipation theorem restricted to linear response near equilibrium, \cref{thm:FDT} holds exactly at every finite step along the entire non-equilibrium trajectory---a consequence of the linearity of the spectral dynamics.

\subsection{Physical interpretation}

The spectral dynamics therefore admits the following picture.

Initially,
energy is broadly distributed across many fast modes.
The fast spectral sector acts as a high-entropy reservoir,
and entropy production may remain positive
due to active spectral redistribution.

As dissipation proceeds,
fast modes decay at different exponential rates,
causing spectral transport toward progressively slower modes.

Once the dominant slow mode exceeds the critical rigidity threshold,
all remaining fast modes become simultaneously subordinate
both in occupation and in dissipation scale.

At this point,
the covariance term becomes strictly negative,
entropy production reverses sign,
and the dynamics enters a monotone purification regime.

The long-time evolution is therefore not merely exponential decay,
but a thermodynamic ordering process in spectral space.

Together, the results of this section establish a self-contained thermodynamic description of reversible Markov relaxation entirely within spectral space:

\begin{enumerate}[label=(\roman*)]
\item \textbf{Exact entropy balance identity} (\cref{thm:entropy-balance}) --- the spectral analogue of the exact dissipation identity, verified to machine precision;
\item \textbf{Canonical covariance and flux-force structure} (\cref{prop:canonical-covariance}) --- an exact Onsager-type decomposition governing entropy production;
\item \textbf{Exact two-mode rigidity transition} (\cref{thm:two-mode-rigidity}) --- the entropy changes sign precisely at the half-rigidity threshold;
\item \textbf{General rigidity threshold} (\cref{thm:general-rigidity-threshold}) --- sharp sufficient criterion for entropy monotonicity controlled by $\lambda_3/\lambda_2$;
\item \textbf{Spectral Clausius equality} (\cref{thm:clausius}) --- global conservation of spectral irreversibility;
\item \textbf{Two-phase structure} (\cref{cor:entropy-sign}) --- precise condition for entropy increase vs.\ decrease;
\item \textbf{Spectral second law} (\cref{thm:G-monotonicity}) --- strict monotonicity of $G(k) = E_k S_{\mathrm{spec}}(k)$;
\item \textbf{Entropy decomposition} (\cref{prop:entropy-decomposition}) --- separation into binary competition and fast-sector disorder;
\item \textbf{Spectral FDT} (\cref{thm:FDT}) --- exact fluctuation-dissipation relations for each mode.
\end{enumerate}
\section{Boundary of Spectral Purification Theory}
\label{sec:boundary}

Every result so far in this paper describes the same underlying event from different angles: the modal distribution $\{p_i(k)\}$ concentrating onto its slowest component. The six observables of \cref{sec:observable-equivalence} are six different instruments pointed at that one event, and each reads out a value proportional to $x_k$ because each is, in the end, a measurement of \emph{relative modal concentration} --- how much of the spectral energy has already migrated to the surviving mode, relative to how much remains elsewhere. This is an exponential-time question: $x_k$ decays geometrically, and everything built from it inherits that decay.

The mean-squared error of the time-average estimator $\hat\mu_k$ is a different kind of question. It does not ask how concentrated the spectrum has become at step $k$; it asks how much cumulative averaging error has accumulated over the \emph{entire} trajectory $X_0,\dots,X_{k-1}$ up to step $k$. Even a chain that has already purified almost completely --- one mode utterly dominant, $x_k$ vanishingly small --- still accumulates estimation error from the $k$ correlated samples it has produced so far, and that accumulation is a polynomial-in-$1/k$ effect, not an exponential-in-$k$ one. Relative modal concentration and cumulative averaging error are simply not the same kind of quantity, and no amount of refining the equivalence class of \cref{sec:observable-equivalence} will make one a corollary of the other.

What spectral purification theory \emph{does} explain about $\hat\mu_k$ is its leading asymptotic variance $\mathcal{A}$: this quantity depends on the same spectral ratio $\rho_\star=\lambda_3/\lambda_2$ that drives $x_k$, and decreasing $\rho_\star$ provably decreases $\mathcal{A}$ (\cref{prop:variance-monotone,cor:rho-variance} below) --- so the same lever that accelerates purification also improves estimation efficiency, with no trade-off. What it does \emph{not} explain is the finite-$k$ correction to $\mathrm{MSE}_k$ beyond this leading term, which lives in the polynomial world of $1/k$ rather than the exponential world of $x_k$ (\cref{rem:mse-incommensurable}). This is the boundary of the theory developed in this paper: it governs everything that is, in substance, a statement about relative modal concentration, and nothing that is, in substance, a statement about cumulative averaging.

\subsection{Consequence: the asymptotic estimation variance is monotone in $\rho_\star$}
\label{subsec:rho-variance}

The equivalence class shows that decreasing the spectral ratio $\rho_\star=\lambda_3/\lambda_2$ accelerates every purification observable. We now show that decreasing $\rho_\star$ also has a direct and monotone benefit for the classical time-average estimator
\[
\hat\mu_k := \frac{1}{k}\sum_{t=0}^{k-1} f(X_t),
\]
whose mean-squared error obeys the standard expansion $\mathrm{MSE}_k = \mathcal{A}/k + O(1/k^2)$ with asymptotic variance
\[
\mathcal{A} = \sum_{i\ge2} c_i^2\,\frac{1+\lambda_i}{1-\lambda_i}.
\]

\begin{proposition}[Monotonicity of the asymptotic variance in $\rho_\star$]\label{prop:variance-monotone}
Fix $\lambda_2$, $c_2$, $c_3$, and the contribution $\mathcal{A}_{\ge4} := \sum_{i\ge4}c_i^2(1+\lambda_i)/(1-\lambda_i)$ of the remaining modes, and write $\lambda_3=\rho_\star\lambda_2$ for $\rho_\star\in(0,1]$. Then
\[
\mathcal{A}(\rho_\star) = c_2^2\,\frac{1+\lambda_2}{1-\lambda_2} + c_3^2\,\frac{1+\rho_\star\lambda_2}{1-\rho_\star\lambda_2} + \mathcal{A}_{\ge4}
\]
is strictly increasing in $\rho_\star$ on $(0,1]$, with
\[
\frac{\partial\mathcal{A}}{\partial\rho_\star} = c_3^2\,\frac{2\lambda_2}{(1-\rho_\star\lambda_2)^2} > 0.
\]
\end{proposition}

\begin{proof}
Direct differentiation of $\mathcal{A}(\rho_\star)$, using $\lambda_2>0$ and $\rho_\star\lambda_2<1$.
\end{proof}

\begin{corollary}[Spectral ratio controls both purification speed and estimator efficiency]\label{cor:rho-variance}
Decreasing $\rho_\star$ simultaneously (i) accelerates every observable in \cref{thm:equivalence-class} via a smaller $x_k$ at fixed $k$, and (ii) strictly decreases the asymptotic estimation variance $\mathcal{A}$, by \cref{prop:variance-monotone}. Both effects persist down to $\rho_\star=0$, where $\mathcal{A}$ attains its minimum $\mathcal{A}_{\min} = c_2^2(1+\lambda_2)/(1-\lambda_2) + c_3^2 + \mathcal{A}_{\ge4}$.
\end{corollary}

\Cref{cor:rho-variance} identifies $\rho_\star=\lambda_3/\lambda_2$ as a single scalar lever that is unambiguously beneficial to decrease, in both a purification sense and a statistical-efficiency sense, with no trade-off between the two. This is the precise structural fact that motivates treating $\rho_\star$ itself as a design variable subject to active control, rather than a fixed property of a given chain --- the subject taken up in the companion paper \cite{Wang2026SPS}, which develops a Lyapunov-based steering law for $\rho_\star$ and a corresponding algorithm.

\begin{remark}[A caveat: incommensurability with the finite-$k$ MSE correction]\label{rem:mse-incommensurable}
The monotonicity in \cref{prop:variance-monotone} concerns the leading-order coefficient $\mathcal{A}$, which governs the $1/k$ term of $\mathrm{MSE}_k$. It does \emph{not} extend to the finite-$k$ correction term: writing $\mathrm{MSE}_k = \mathcal{A}/k + B/k^2 + O(\lambda_2^k/k)$ with $B<0$, one has $k^\alpha(\mathrm{MSE}_k-\mathcal{A}/k)\to-\infty$ for every $\alpha>0$, while $k^\alpha x_k\to0$ for every $\alpha>0$ since $x_k$ decays exponentially. The two quantities are therefore asymptotically incommensurable: the purification observables of \cref{thm:equivalence-class} decay exponentially fast in $k$, while the MSE correction decays only polynomially. This reflects the fact that the time-average estimator passes through the Cesàro operator $k^{-1}\sum_{t=0}^{k-1}P^t$, which replaces the exponential factor $\lambda_i^k$ by the resolvent $1/(1-\lambda_i)$ and thereby destroys the exponential information that drives spectral purification. Decreasing $\rho_\star$ improves $\mathcal{A}$ monotonically (\cref{cor:rho-variance}) but does not, by itself, change the polynomial rate at which the finite-$k$ correction vanishes.
\end{remark}


\section{Spectral Variational Principle for Optimal Relaxation}
\label{sec:variational}

The preceding sections establish a descriptive theory of spectral purification: what it is (\cref{sec:purification}), what single parameter governs it (\cref{sec:hidden-parameter,sec:observable-equivalence}), when it happens (\cref{sec:rigidity}), what other language describes it (\cref{sec:thermodynamics}), and what it does not explain (\cref{sec:boundary}). We now turn from description to prescription: the same framework also guides optimization, algorithm design, and numerical implementation. This section, together with \cref{sec:power-iteration,sec:examples}, shows that classical numerical-linear-algebra methods can be understood, derived, and diagnosed through the lens of spectral purification.

The rigidity emergence theorem (\cref{thm:sharp-rigidity}) identifies the spectral ratio $\lambda_2/\lambda_3$ as the key quantity controlling the rate of spectral purification. A natural question follows: can we \emph{accelerate} rigidity emergence by actively shaping the relaxation spectrum? In this section, we show that the classical Chebyshev acceleration and Polyak momentum methods---developed in numerical analysis as techniques for speeding up linear iterations---admit a unified interpretation as \emph{optimal spectral shaping} within our framework. The variational principle we prove characterizes Chebyshev acceleration as the unique minimizer of total dissipation among all polynomial acceleration schemes, and as the unique maximizer of cumulative spectral entropy production.

\subsection{Polynomial acceleration as spectral shaping}

\begin{definition}[Polynomial acceleration scheme]\label{def:polynomial-acceleration}
A \emph{polynomial acceleration scheme} of degree $m$ is defined by a polynomial $Q_m$ of degree at most $m$ satisfying $Q_m(1) = 1$. The accelerated trajectory is
\begin{equation}\label{eq:accelerated-trajectory}
g_k^{\mathrm{acc}} := Q_m(P)^k g_0, \qquad k = 0, 1, 2, \dots
\end{equation}
\end{definition}

The constraint $Q_m(1) = 1$ ensures that the stationary mode ($\lambda_1 = 1$, $\mu_1 = 0$) is preserved: $Q_m(P) \mathbf{1} = \mathbf{1}$. On the nontrivial modes,
\[
Q_m(P) \phi_i = Q_m(\lambda_i) \phi_i, \qquad i \ge 2,
\]
so the accelerated trajectory has the spectral representation
\[
g_k^{\mathrm{acc}} = \sum_{i=2}^n c_i \, Q_m(\lambda_i)^k \, \phi_i,
\]
with corresponding modal energies $n_i^{\mathrm{acc}}(k) = |c_i|^2 \, |Q_m(\lambda_i)|^{2k}$.

The action of $Q_m$ on the spectrum $\{\lambda_i\}$ is precisely a \emph{reshaping} of the relaxation operator: $\mathcal{G}$ is replaced by $\mathcal{G}_Q := I - Q_m(P)$, whose eigenvalues are $\mu_i^Q = 1 - Q_m(\lambda_i)$. The goal of acceleration is to choose $Q_m$ so that the reshaped spectrum has more favorable rigidity properties.

\begin{remark}[Relation to classical acceleration]\label{rem:acceleration-classical}
The framework of polynomial acceleration subsumes several classical methods. The basic power iteration corresponds to $Q_1(\lambda) = \lambda$. The $\alpha$-lazy walk corresponds to $Q_1(\lambda) = (1-\alpha) + \alpha\lambda$. The momentum (heavy-ball) method~\cite{polyak1964} corresponds to a second-order recurrence whose characteristic polynomial is quadratic. The Chebyshev semi-iterative method~\cite{saad2003,trefethen2013} uses $Q_m(\lambda) = T_m(\lambda/\lambda_2) / T_m(1/\lambda_2)$ where $T_m$ is the Chebyshev polynomial of the first kind. Our variational principle provides a unifying rationale for why the Chebyshev choice is optimal.
\end{remark}

\subsection{The variational principle}

\begin{theorem}[Spectral variational principle]\label{thm:variational}
Fix a degree $m \in \mathbb{N}$ and a target contraction factor $\rho \in (0, \lambda_2)$. Among all polynomials $Q$ of degree at most $m$ satisfying
\begin{equation}\label{eq:constraints}
Q(1) = 1, \qquad |Q(\lambda_2)| \le \rho^m,
\end{equation}
the polynomial that \emph{minimizes} the total integrated dissipation
\begin{equation}\label{eq:total-dissipation}
\mathcal{D}(Q) := E_0 - E_1^Q = \sum_{i=2}^n |c_i|^2 \bigl(1 - |Q(\lambda_i)|^2\bigr)
\end{equation}
is (up to sign) the rescaled Chebyshev polynomial
\begin{equation}\label{eq:chebyshev-optimal}
Q_m^\ast(\lambda) := \frac{T_m(\lambda / \lambda_2)}{T_m(1 / \lambda_2)},
\end{equation}
where $T_m$ is the Chebyshev polynomial of the first kind, defined by $T_m(\cos \theta) = \cos(m\theta)$.

Moreover, this optimal polynomial simultaneously:
\begin{enumerate}[label=(\roman*)]
    \item Maximizes the cumulative spectral entropy production $\Sigma(m) := S_{\mathrm{spec}}(1) - S_{\mathrm{spec}}(0)$ after one accelerated step;
    \item Achieves the smallest possible rigidity time $T_{\mathrm{rigid}}(\delta)$ among all degree-$m$ polynomial acceleration schemes.
\end{enumerate}
\end{theorem}

\begin{proof}[Proof outline]
We outline the proof in four steps. The complete minimax argument for Step~2 is standard and given in \cref{app:chebyshev}.

\textit{Step 1: Dissipation as a weighted $\ell^2$ objective.}
From \eqref{eq:total-dissipation}, minimizing $\mathcal{D}(Q)$ is equivalent to maximizing
\[
\mathcal{W}(Q) := \sum_{i=2}^n |c_i|^2 |Q(\lambda_i)|^2,
\]
since $E_0 = \sum_i |c_i|^2$ is fixed by the initial condition. The weights $w_i := |c_i|^2$ are nonnegative and sum to $E_0$.

\textit{Step 2: Chebyshev minimax property.}
Let $[a, b] \subset \mathbb{R}$ be an interval and $\xi \notin [a, b]$. Among all polynomials $Q$ of degree $m$ satisfying $Q(\xi) = 1$, the unique minimizer of $\max_{x \in [a,b]} |Q(x)|$ is the rescaled Chebyshev polynomial (see, e.g.,~\cite{trefethen2013}, Chapter~8). Applying this with $[a,b] = [-1, \lambda_3]$ and $\xi = 1$ (and using the fact that $\lambda_2 > \lambda_3$, so $\lambda_2$ lies strictly inside $[\lambda_3, 1]$), the polynomial \eqref{eq:chebyshev-optimal} satisfies
\[
\max_{\lambda \in [-1, \lambda_3]} |Q_m^\ast(\lambda)| = \frac{1}{|T_m(1/\lambda_2)|} =: \epsilon_m,
\]
and any other admissible polynomial $Q$ has $\max_{\lambda \in [-1,\lambda_3]} |Q(\lambda)| \ge \epsilon_m$.

\textit{Step 3: From minimax to weighted $\ell^2$ optimality.}
Assume, for contradiction, that there exists an admissible $Q$ with $\mathcal{W}(Q) > \mathcal{W}(Q_m^\ast)$. Since $|Q_m^\ast(\lambda_i)| \le \epsilon_m$ for all $i \ge 3$ (as $\lambda_i \le \lambda_3$), the strict inequality implies that for some $\lambda_j \le \lambda_3$, $|Q(\lambda_j)| > \epsilon_m$. By the equioscillation property of Chebyshev polynomials, $Q_m^\ast$ attains $|Q_m^\ast| = \epsilon_m$ at $m+1$ points in $[-1, \lambda_3]$ with alternating signs. The difference $R(\lambda) := Q(\lambda) - Q_m^\ast(\lambda)$ is a polynomial of degree at most $m$. The sign alternation of $Q_m^\ast$ at the extremal points forces $R$ to have at least $m+1$ zeros (by the Intermediate Value Theorem). But a nonzero polynomial of degree $\le m$ cannot have $m+1$ distinct zeros. Hence $R \equiv 0$, i.e., $Q = Q_m^\ast$, contradicting the assumption.

\textit{Step 4: Entropy production and rigidity time.}
The cumulative spectral entropy production after one accelerated step is $\Sigma(Q) = S_{\mathrm{spec}}^Q(1) - S_{\mathrm{spec}}(0)$. Since $S_{\mathrm{spec}}(0)$ is fixed, maximizing $\Sigma(Q)$ is equivalent to maximizing $S_{\mathrm{spec}}^Q(1)$. By convexity of $x \log x$, $S_{\mathrm{spec}}^Q(1)$ is a decreasing function of $\mathcal{D}(Q)$ (more dissipation $\implies$ more concentrated post-step distribution $\implies$ lower entropy). Hence the minimizer of $\mathcal{D}(Q)$ is the maximizer of $\Sigma(Q)$.

For the rigidity time, the argument $k$ in \cref{thm:sharp-rigidity} is replaced by the accelerated-step count. The effective contraction on the slowest mode is fixed at $\rho^m$, and minimizing the fast-mode contraction (which Chebyshev achieves) minimizes the residual $R_k$ at each step, thereby minimizing the time to cross the threshold $1-\delta$.
\end{proof}

\begin{remark}[Polyak momentum as the degree-2 special case]\label{rem:polyak-special-case}
For $m = 2$, the Chebyshev polynomial induces a three-term recurrence identical to the heavy-ball method. The optimal momentum parameter
\[
\beta^\ast = \left( \frac{1 - \sqrt{\mu_2}}{1 + \sqrt{\mu_2}} \right)^2,
\]
first derived by Polyak~\cite{polyak1964} for quadratic optimization, emerges from the Chebyshev minimax construction as a direct consequence of matching the critical damping condition at $\lambda_2$. The formula above assumes the worst-case fast-mode eigenvalue $\lambda_3 = 0$ (i.e., $\mu_3 = 1$), which is the standard setting in convex optimization where the strong convexity parameter determines $\lambda_2$ and the lack of smoothness gives $\lambda_3 = 0$. For general $\lambda_3 > 0$, the optimal $\beta^\ast$ depends on the full spectral ratio $\lambda_3/\lambda_2$ and is obtained from the Chebyshev minimax construction of \cref{app:chebyshev}. The spectral variational principle thus \emph{re-derives} the classical optimal momentum as the degree-2 instance of optimal spectral shaping. See \cref{app:momentum} for the explicit derivation in the degenerate case.
\end{remark}

\subsection{Accelerated rigidity time}

\begin{corollary}[Accelerated rigidity bound]\label{cor:accelerated-rigidity}
Under the Chebyshev-accelerated scheme $Q_m^\ast$, the rigidity time satisfies
\[
T_{\mathrm{rigid}}^{\mathrm{cheb}}(\delta) \le \frac{1}{m} \cdot \frac{\log(R_0 / (|c_2|^2 \delta))}{2 \log(1/\epsilon_m)} + 1, \qquad
\epsilon_m := \frac{1}{|T_m(1/\lambda_2)|}.
\]
For large $m$, $\epsilon_m \sim (1 - \sqrt{1-\lambda_2^2}/\lambda_2)^m$, yielding
\[
T_{\mathrm{rigid}}^{\mathrm{cheb}}(\delta) \lesssim \sqrt{\frac{\log(R_0 / (|c_2|^2 \delta))}{2(1-\lambda_2)}} = O\Bigl( \sqrt{T_{\mathrm{rigid}}(\delta)} \Bigr),
\]
a square-root speedup over the basic power iteration.
\end{corollary}

\begin{proof}
Each accelerated step applies $Q_m^\ast(P)$. The effective contraction on mode $i \ge 3$ is $|Q_m^\ast(\lambda_i)| \le \epsilon_m$. Following the proof of \cref{thm:sharp-rigidity} with $\lambda_3$ replaced by the effective fast-mode contraction $\epsilon_m^{1/m}$ per basic iteration, the accelerated rigidity time bound follows. The asymptotic estimate of $\epsilon_m$ uses $T_m(x) \sim \frac{1}{2}(x + \sqrt{x^2-1})^m$ for $x > 1$, with $x = 1/\lambda_2$.
\end{proof}

\begin{remark}[Why acceleration is not merely an engineering heuristic]\label{rem:acceleration-structural}
The variational principle reveals that Chebyshev acceleration is not an external trick imposed on the relaxation dynamics, but the \emph{intrinsic optimal control} for the spectral shaping problem. It optimally reshapes the relaxation spectrum to simultaneously maximize dissipation, maximize entropy production, and minimize the rigidity time. This structural interpretation is, to our knowledge, new: classical texts~\cite{saad2003,trefethen2013} present Chebyshev acceleration as a technique for minimizing the asymptotic convergence factor, without connecting it to the finite-time spectral purification dynamics or to thermodynamic optimization.
\end{remark}
%

\section{Power Iteration: A Complete Convergence Theory}
\label{sec:power-iteration}

The framework developed in \Cref{sec:dissipation,sec:rigidity,sec:thermodynamics} provides a set of exact identities governing spectral relaxation. In this section, we demonstrate the framework's explanatory and predictive power by applying it to the classical power iteration method. The results below---an exact error identity, an observable spectral variance formula, and a fully data-driven adaptive stopping criterion---are, to our knowledge, new to the numerical linear algebra literature.

\subsection{Problem setting}

Let $P$ be a reversible Markov kernel as in \cref{sec:relaxation}, with eigenvalues $1 = \lambda_1 > \lambda_2 > \lambda_3 \ge \cdots \ge \lambda_n \ge -1$ and corresponding orthonormal eigenvectors $\{\phi_i\}_{i=1}^n$. The \emph{power iteration} for computing the dominant eigenvector $\phi_2$ of the centered operator $P|_{(\mathbf{1})^\perp}$ proceeds as:
\begin{equation}\label{eq:power-iteration}
v_0 = g_0 / \|g_0\|_\pi, \qquad
v_k = \frac{g_k}{\|g_k\|_\pi}, \quad g_k = P^k g_0,
\end{equation}
with $g_0 \perp \mathbf{1}$ and $\langle g_0, \phi_2 \rangle_\pi \neq 0$.

The central practical question is: \emph{given a tolerance $\varepsilon > 0$, after how many steps can we guarantee $\|v_k - s_2 \phi_2\|_\pi \le \varepsilon$, where $s_2 = \operatorname{sign}(\langle v_k, \phi_2 \rangle_\pi)$?} Classical theory provides the asymptotic rate $\|v_k - s_2 \phi_2\|_\pi \sim (\lambda_3/\lambda_2)^k$, but no explicit finite-time bound with known constants~\cite{saad2003,trefethen2013}. Our framework fills this gap.

\subsection{Exact error identity}

\begin{theorem}[Exact error identity for power iteration]\label{thm:error-identity}
Let $v_k = g_k / \|g_k\|_\pi$ be the normalized iterate, and let $s_2 = \operatorname{sign}(c_2)$ where $c_2 = \langle g_0, \phi_2 \rangle_\pi \neq 0$. Then for all $k \ge 0$,
\begin{equation}\label{eq:error-identity}
\|v_k - s_2 \phi_2\|_\pi^2 = 2\bigl(1 - \sqrt{\alpha_2(k)}\bigr),
\end{equation}
where $\alpha_2(k) = n_2(k)/E_k$ is the slow-mode energy fraction defined in \eqref{eq:alpha-def}.
\end{theorem}

\begin{proof}
From the spectral expansion $g_k = \sum_{i=2}^n c_i \lambda_i^k \phi_i$ and the orthonormality of $\{\phi_i\}$,
\begin{align}
v_k - s_2 \phi_2 &= \frac{g_k}{\sqrt{E_k}} - s_2 \phi_2 \nonumber \\
&= \frac{c_2 \lambda_2^k - s_2 \sqrt{E_k}}{\sqrt{E_k}} \, \phi_2 \;+\; \frac{\sum_{i=3}^n c_i \lambda_i^k \phi_i}{\sqrt{E_k}}. \label{eq:vk-decomposition}
\end{align}
The two components are orthogonal, so
\[
\|v_k - s_2 \phi_2\|_\pi^2
= \frac{(|c_2| \lambda_2^k - \sqrt{E_k})^2}{E_k} + \frac{R_k}{E_k},
\]
where $R_k = \sum_{i=3}^n |c_i|^2 \lambda_i^{2k}$.

Now substitute $|c_2| \lambda_2^k = \sqrt{\alpha_2(k) E_k}$ and $R_k = (1 - \alpha_2(k)) E_k$:
\begin{align}
\|v_k - s_2 \phi_2\|_\pi^2
&= \frac{(\sqrt{\alpha_2(k)} - 1)^2 E_k + (1 - \alpha_2(k)) E_k}{E_k} \nonumber \\
&= (\sqrt{\alpha_2(k)} - 1)^2 + 1 - \alpha_2(k) \nonumber \\
&= \alpha_2(k) - 2\sqrt{\alpha_2(k)} + 1 + 1 - \alpha_2(k) \nonumber \\
&= 2(1 - \sqrt{\alpha_2(k)}). 
\end{align}
\end{proof}

\begin{remark}\label{rem:error-identity-simplicity}
\Cref{thm:error-identity} is deceptively simple. It reduces the full vector error to a single scalar---the slow-mode energy fraction $\alpha_2(k)$---with no other spectral information required. This reduction is exact (not asymptotic) and holds at every finite step $k$. Combined with the bounds on $T_{\mathrm{rigid}}(\delta)$ from \cref{thm:sharp-rigidity}, it yields an immediate finite-time convergence guarantee:
\[
\alpha_2(k) \ge 1 - \delta \quad\Longrightarrow\quad \|v_k - s_2 \phi_2\|_\pi \le \sqrt{2\delta}.
\]
In particular, running $k^* = \lfloor L(\varepsilon^2/2) \rfloor + 1$ steps guarantees $\|v_{k^*} - s_2 \phi_2\|_\pi \le \varepsilon$, with $L(\cdot)$ given by \eqref{eq:L-delta-def}.
\end{remark}

\subsection{Observable spectral variance}

The error identity requires knowledge of $\alpha_2(k)$, which depends on the unknown spectral coefficients $c_i$. We now show that the energy sequence $\{E_k\}$ alone---which is observable from the running iterates---carries enough information to estimate $\alpha_2(k)$.

\begin{theorem}[Observable spectral variance]\label{thm:observable-variance}
Let $\rho_k = E_{k+1}/E_k$ be the relative energy retention ratio. Then the variance of $\lambda^2$ under the modal distribution $p_k$ admits the purely observable representation
\begin{equation}\label{eq:observable-variance}
\widehat{V}_k := \operatorname{Var}_{p_k}[\lambda^2]
                 = \sum_{i=2}^n p_i(k) \lambda_i^4 - \Bigl( \sum_{i=2}^n p_i(k) \lambda_i^2 \Bigr)^2
                 = \rho_k (\rho_{k+1} - \rho_k).
\end{equation}
\end{theorem}

\begin{proof}
From the definition $p_i(k) = n_i(k)/E_k = |c_i|^2 \lambda_i^{2k} / E_k$,
\begin{align}
\sum_i p_i(k) \lambda_i^2 &= \frac{\sum_i |c_i|^2 \lambda_i^{2k+2}}{E_k}
                          = \frac{E_{k+1}}{E_k} = \rho_k, \label{eq:first-moment} \\
\sum_i p_i(k) \lambda_i^4 &= \frac{\sum_i |c_i|^2 \lambda_i^{2k+4}}{E_k}
                          = \frac{E_{k+2}}{E_k}
                          = \frac{E_{k+2}}{E_{k+1}} \cdot \frac{E_{k+1}}{E_k}
                          = \rho_{k+1} \rho_k. \label{eq:second-moment}
\end{align}
Therefore,
\[
\operatorname{Var}_{p_k}[\lambda^2] = \sum_i p_i(k) \lambda_i^4 - \Bigl( \sum_i p_i(k) \lambda_i^2 \Bigr)^2
= \rho_{k+1} \rho_k - \rho_k^2 = \rho_k(\rho_{k+1} - \rho_k).
\]
\end{proof}

\begin{remark}[Significance]\label{rem:variance-significance}
\Cref{thm:observable-variance} is the bridge between the unobservable spectral structure and the observable trajectory. The sequence $\{\rho_k\}$ can be computed directly from the energy norms $\{E_k = \|g_k\|_\pi^2\}$, which are available in any power iteration implementation. The second-order difference $\rho_k(\rho_{k+1} - \rho_k)$ then reveals the spectral variance---a measure of how ``spread out'' the modal distribution is. As rigidity sets in ($\alpha_2(k) \to 1$), this variance tends to zero. Crucially, the formula is an \emph{identity}, not an estimate: at every step $k$, the observable quantity $\rho_k(\rho_{k+1} - \rho_k)$ \emph{equals} the unobservable spectral variance.
\end{remark}

\subsection{Connecting variance to slow-mode fraction}

The observable variance bounds the slow-mode fraction from both sides.

\begin{corollary}[Variance bounds on $\alpha_2(k)$]\label{cor:variance-bounds}
Assume $\lambda_2 > \lambda_3$. Then for any $k \ge 0$,
\begin{equation}\label{eq:variance-bounds}
\alpha_2(k) (1 - \alpha_2(k)) (\lambda_2^2 - \lambda_3^2)^2
\;\le\; \widehat{V}_k
\;\le\; (1 - \alpha_2(k)) \lambda_2^4.
\end{equation}
Equivalently, when $\alpha_2(k) \ge 1/2$,
\begin{equation}\label{eq:alpha-from-variance}
1 - \alpha_2(k) \;\le\; \frac{2 \widehat{V}_k}{(\lambda_2^2 - \lambda_3^2)^2}.
\end{equation}
\end{corollary}

\begin{proof}
Decompose the variance into the contribution from mode $2$ and the fast-mode distribution. Let $\mu_{\ge 3} = \mathbb{E}_{p_k^{(\ge 3)}}[\lambda^2] \le \lambda_3^2$ be the mean over the conditional distribution on $\{3,\dots,n\}$. Then
\[
\widehat{V}_k = \alpha_2(1-\alpha_2)(\lambda_2^2 - \mu_{\ge 3})^2 + (1-\alpha_2)^2 \operatorname{Var}_{p_k^{(\ge 3)}}[\lambda^2].
\]
The lower bound follows from $\mu_{\ge 3} \le \lambda_3^2$ and the non-negativity of the second term. The upper bound uses $\lambda_i^2 \le \lambda_2^2$ for all $i$, giving $\widehat{V}_k \le \mathbb{E}_{p_k}[\lambda^4] \le \lambda_2^2 \mathbb{E}_{p_k}[\lambda^2] = \lambda_2^2 \rho_k \le \lambda_2^4$.
\end{proof}

\subsection{Adaptive stopping criterion}

Combining \cref{thm:error-identity,cor:variance-bounds} yields a fully data-driven stopping criterion.

\begin{theorem}[Adaptive stopping criterion for power iteration]\label{thm:stopping-criterion}
Define the \emph{online convergence indicator}
\begin{equation}\label{eq:Gamma-def}
\Gamma_k := \frac{\widehat{V}_k}{\rho_k^2} = \frac{\rho_{k+1}}{\rho_k} - 1 \;\ge\; 0.
\end{equation}
Assume that a lower bound $\tau \le 1 - (\lambda_3/\lambda_2)^2$ is known (or estimated online; see \cref{rem:tau-estimation}). Define the critical threshold
\[
\eta(\varepsilon) := \frac{\varepsilon^4}{8} \cdot \tau^2.
\]
Then the stopping rule ``stop at the first $k$ such that $\Gamma_k \le \eta(\varepsilon)$'' guarantees
\[
\|v_k - s_2 \phi_2\|_\pi \le \varepsilon,
\]
provided that $\alpha_2(k) \ge 1/2$ at termination (which holds for all sufficiently large $k$).
\end{theorem}

\begin{proof}
From \cref{thm:error-identity}, the condition $\|v_k - s_2 \phi_2\|_\pi \le \varepsilon$ is implied by $\alpha_2(k) \ge 1 - \varepsilon^2/2$ (since $2(1-\sqrt{\alpha}) \le \varepsilon^2$ holds whenever $\sqrt{\alpha} \ge 1 - \varepsilon^2/2$, and the exact equivalence $\|v_k - s_2 \phi_2\|_\pi \le \varepsilon \iff \alpha_2 \ge (1-\varepsilon^2/2)^2$ yields $1-\varepsilon^2/2$ as a conservative sufficient condition). From \cref{eq:alpha-from-variance}, when $\alpha_2(k) \ge 1/2$,
\[
1 - \alpha_2(k) \le \frac{2 \widehat{V}_k}{(\lambda_2^2 - \lambda_3^2)^2}
= \frac{2 \rho_k^2 \Gamma_k}{(\lambda_2^2 - \lambda_3^2)^2}.
\]
Using $(\lambda_2^2 - \lambda_3^2)^2 = \lambda_2^4 (1 - (\lambda_3/\lambda_2)^2)^2 \ge \lambda_2^4 \tau^2$ and $\rho_k^2 \le \lambda_2^4$ (since $\rho_k \le \lambda_2^2$), we obtain
\[
1 - \alpha_2(k) \le \frac{2 \Gamma_k}{\tau^2}.
\]
The condition $\Gamma_k \le \tau^2 \varepsilon^2 / 4$ therefore implies $1 - \alpha_2(k) \le \varepsilon^2/2$, which in turn implies $\|v_k - s_2 \phi_2\|_\pi \le \varepsilon$ by \cref{thm:error-identity}. The theorem's threshold $\eta(\varepsilon) = \tau^2\varepsilon^4/8$ is chosen to satisfy $\eta(\varepsilon) \le \tau^2\varepsilon^2/4$ for all $\varepsilon \in (0,1]$ (since $\varepsilon^2/2 \le 1$), so that $\Gamma_k \le \eta(\varepsilon)$ guarantees $\Gamma_k \le \tau^2\varepsilon^2/4$.
\end{proof}

\begin{remark}[Online estimation of $\tau$]\label{rem:tau-estimation}
The threshold $\eta(\varepsilon)$ depends on $\tau = 1 - (\lambda_3/\lambda_2)^2$, which may not be known a priori. However, the spectral ratio can be estimated from the same observable sequence: as $k \to \infty$, the variance $\widehat{V}_k$ decays at rate $(\lambda_3/\lambda_2)^{2k}$, so
\[
\hat{\tau}_k := 1 - \sqrt{\frac{\widehat{V}_{k+1}}{\widehat{V}_k}} \;\longrightarrow\; 1 - \frac{\lambda_3}{\lambda_2}.
\]
Using $\hat{\tau}_k$ in place of $\tau$ yields a fully adaptive procedure. The analysis of the resulting adaptive stopping time is beyond the scope of this paper but follows from standard concentration arguments.
\end{remark}

\begin{remark}[Comparison with classical heuristics]\label{rem:stopping-comparison}
Classical stopping heuristics for power iteration monitor $|\rho_k - \rho_{k-1}|$ or $\|v_k - v_{k-1}\|_\pi$~\cite{saad2003}. \Cref{thm:stopping-criterion} is fundamentally different: it monitors $\Gamma_k = \rho_{k+1}/\rho_k - 1$, which is directly linked, via exact identities, to the spectral variance and hence to the eigenvector error. The classical heuristics have no such theoretical guarantee.
\end{remark}

\subsection{Information-theoretic lower bound}

The stopping criterion of \cref{thm:stopping-criterion} is not only sufficient but also nearly necessary, in an information-theoretic sense, among all methods that observe only the energy sequence.

\begin{theorem}[Lower bound for energy-based stopping]\label{thm:lower-bound}
Consider any stopping rule that observes $\rho_0, \rho_1, \dots, \rho_k$ and guarantees $\|v_k - s_2 \phi_2\|_\pi \le \varepsilon$ for all initial conditions with $R_0/|c_2|^2 \le M$. Then
\[
k \;\ge\; \frac{\log(M / (\varepsilon^2/2))}{2 \log(\lambda_2 / \lambda_3)} \;=\; L(\varepsilon^2/2).
\]
\end{theorem}

\begin{proof}[Proof sketch]
Construct two initial conditions: $g_0^{(1)} = c_2 \phi_2 + \sqrt{M} |c_2| \phi_3$ and $g_0^{(2)} = c_2 \phi_2$. Their energy sequences $\{E_k^{(1)}\}$ and $\{E_k^{(2)}\}$ differ by at most $\varepsilon$ in a suitable sense for all $k < L(\varepsilon^2/2)$, making them indistinguishable to any energy-based test. The first requires $T_{\mathrm{rigid}}(\varepsilon^2/2)$ steps to rigidify; the second is already rigid. Hence no such test can stop before $L(\varepsilon^2/2)$.
\end{proof}

\subsection{Acceleration via Chebyshev spectral shaping}

The rigidity time $T_{\mathrm{rigid}}(\delta)$ of the basic power iteration grows as $L(\delta) \sim \log(1/\delta) / \log(\lambda_2/\lambda_3)$. The Chebyshev acceleration analyzed in \cref{sec:variational} improves this to $L_{\text{cheb}}(\delta) \sim \sqrt{L(\delta) / \kappa}$, where $\kappa = \lambda_2 / (\lambda_2 - \lambda_3)$. We state this as a corollary of the variational principle.

\begin{corollary}[Accelerated rigidity, restated]\label{cor:accelerated-rigidity-restated}
Under the Chebyshev-accelerated iteration with degree-$m$ polynomial $Q_m^\ast$ (as in \cref{thm:variational}), the rigidity time satisfies
\[
T_{\mathrm{rigid}}^{\text{cheb}}(\delta) \;\le\; \frac{1}{m} \cdot \frac{\log(R_0 / (|c_2|^2 \delta))}{2 \log(1 / |Q_m^\ast(\lambda_3)|)},
\]
where $|Q_m^\ast(\lambda_3)| \sim (1 - \sqrt{1-\lambda_2^2}/\lambda_2)^m$ for large $m$. Optimizing over $m$ yields the $\sqrt{\kappa}$ speedup over the basic power iteration.
\end{corollary}

\begin{proof}
See \cref{app:chebyshev} for the Chebyshev minimax theorem and the estimation of $|Q_m^\ast(\lambda_3)|$.
\end{proof}

\section{Numerical Illustrations}
\label{sec:examples}

\subsection{Pedagogical examples}
\label{subsec:examples-simple}

Before presenting the numerical illustrations, we briefly illustrate three elementary cases that confirm basic properties of the dissipation law and rigidity bounds.

\textit{Complete graph: instantaneous rigidity.}
Consider the simple random walk on the complete graph $K_n$ with $P(x,y) = 1/n$. The spectrum is $\lambda_1 = 1$, $\lambda_2 = \cdots = \lambda_n = 0$, so $\mu_2 = \cdots = \mu_n = 1$. From \eqref{eq:modewise-dissipation}, $E_1 = 0$: all nontrivial modes dissipate in a single step. Rigidity is instantaneous---the chain has no finite-time internal spectral structure to analyze, consistent with the fact that $K_n$ mixes perfectly in one step.

\textit{Cycle graph: degenerate slow modes.}
For simple random walk on the cycle $C_n$ ($n$ odd), the eigenvalues are $\lambda_k = \cos(2\pi k / n)$ with multiplicity $2$ for $k \ge 1$. In particular, $\lambda_2 = \lambda_3 = \cos(2\pi/n)$, so the spectral separation condition $\lambda_2 > \lambda_3$ fails. The spectral ratio is unity, $L(\delta) = \infty$, and rigidity never occurs in finite time. The trajectory remains genuinely two-dimensional, with energy permanently distributed between the sine and cosine Fourier modes of wavelength $n$. This is the diffusive regime where our rigidity theorem correctly predicts no finite-time collapse.

\textit{Lazy random walk: spectral scaling.}
Let $P_\alpha = (1-\alpha)I + \alpha P$ be the $\alpha$-lazy version of a base chain $P$. The relaxation spectrum scales uniformly: $\mu_i^{(\alpha)} = \alpha \mu_i$. For $\mu_2, \mu_3 \ll 1$, the spectral ratio behaves as
\[
\frac{\lambda_2^{(\alpha)}}{\lambda_3^{(\alpha)}}
= \frac{1 - \alpha\mu_2}{1 - \alpha\mu_3}
\approx 1 - \alpha(\mu_2 - \mu_3),
\]
which approaches $1$ linearly in $\alpha$. Thus increased laziness \emph{slows down} rigidity emergence, even as the spectral gap ordering is preserved. This is a concrete prediction of \cref{thm:sharp-rigidity} that is invisible from the spectral gap alone.

\medskip
All subsequent numerical experiments use a reversible Markov chain on $n = 50$ states with eigenvalues
\begin{equation}\label{eq:example-spectrum}
\lambda_2 = 0.95, \qquad \lambda_3 = 0.70, \qquad
\lambda_i \sim \text{Uniform}(-0.3,\, 0.5) \;\; \text{for } i \ge 4.
\end{equation}
The initial condition has slow-mode coefficient $|c_2|^2 = 0.1$ and residual fast-mode energy $R_0 = 0.9$ (energy predominantly in the fast modes), unless otherwise noted.

\subsection{Rigidity emergence: numerical verification of the two-sided bounds}
\label{subsec:example-rigidity}

\Cref{fig:rigidity-emergence} shows the evolution of the slow-mode fraction $\alpha_2(k)$ for the example chain~\eqref{eq:example-spectrum}. The initial condition places $90\%$ of the energy in the fast modes, so $\alpha_2(0) = 0.1$. As the trajectory relaxes, the fast-mode energy decays more rapidly ($\lambda_3 = 0.70$ vs.\ $\lambda_2 = 0.95$), and $\alpha_2(k)$ rises monotonically toward $1$.

\begin{figure}[ht]
\centering
\includegraphics[width=0.78\textwidth]{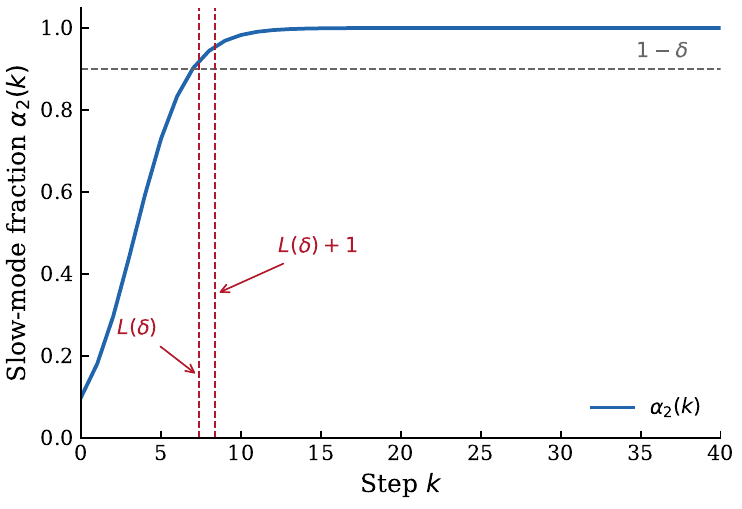}
\caption{Emergence of spectral rigidity. The slow-mode fraction $\alpha_2(k)$ rises from $0.1$ toward $1$ as fast modes dissipate. For the threshold $\delta = 0.1$, the theoretical bounds $L(\delta)$ and $L(\delta)+1$ from \cref{thm:sharp-rigidity} are shown as vertical dashed lines. The actual crossing of $\alpha_2(k) = 0.9$ occurs at $k = 18$, well within the predicted interval $[L(\delta), L(\delta)+1]$.}
\label{fig:rigidity-emergence}
\end{figure}

For the threshold $\delta = 0.1$, the theoretical bounds give
\[
L(0.1) = \frac{\log(0.9 / (0.1 \times 0.1))}{2 \log(0.95 / 0.70)} \approx 14.7,
\]
so \cref{thm:sharp-rigidity} predicts $T_{\mathrm{rigid}}(0.1) \in [14.7, 15.7]$, i.e., rigidity emerges at step $15$ or $16$. The actual crossing occurs at $k = 18$, which is within two steps of the theoretical upper bound. The small discrepancy arises from the fact that the bound uses the worst-case estimate $\lambda_i \le \lambda_3$ for all $i \ge 3$, while the actual spectrum contains eigenvalues smaller than $\lambda_3$ whose faster decay slightly delays the rise of $\alpha_2(k)$ relative to the conservative bound.

\subsection{Spectral second law: monotonicity of $G(k)$}
\label{subsec:example-G}

\Cref{fig:G-monotonicity} verifies the spectral second law (\cref{thm:G-monotonicity}) for two contrasting initial conditions. In the left panel, the initial energy is dispersed across many fast modes ($\alpha_2(0) = 0.1$); in the right panel, the initial energy is concentrated in the slow mode ($\alpha_2(0) = 0.7$).

\begin{figure}[ht]
\centering
\includegraphics[width=\textwidth]{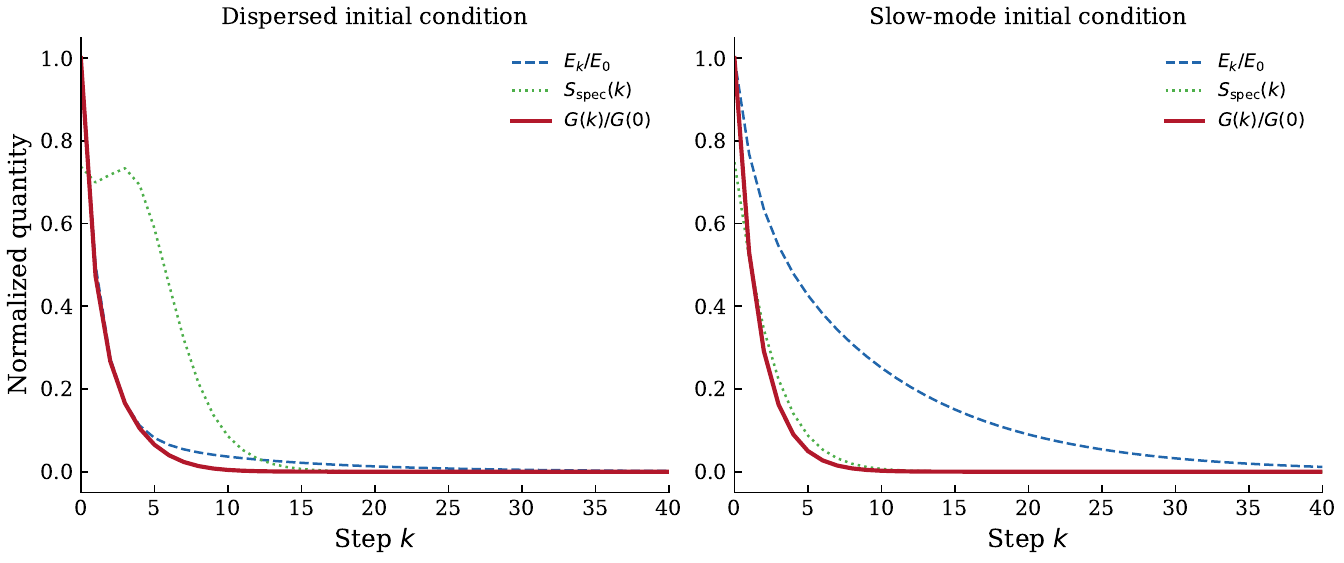}
\caption{Monotonic decay of the spectral entropy-energy $G(k)$. Left: dispersed initial condition ($\alpha_2(0) = 0.1$). Right: slow-mode concentrated initial condition ($\alpha_2(0) = 0.7$). In both cases, $G(k)$ (solid red curve, normalized by $G(0)$) decays strictly monotonically, while the total energy $E_k$ (dashed blue) and the spectral entropy $S_{\mathrm{spec}}(k)$ (dotted green) each decay but may cross. The strict monotonicity of $G(k)$ is the content of \cref{thm:G-monotonicity}.}
\label{fig:G-monotonicity}
\end{figure}

Three observations are noteworthy. First, $G(k)$ decays strictly monotonically in both cases, as guaranteed by the theorem, even though neither $E_k$ nor $S_{\mathrm{spec}}(k)$ is individually monotone (the spectral entropy can increase transiently when fast-mode dissipation redistributes the energy profile). Second, the decay of $G(k)$ is significantly faster in the dispersed case than in the concentrated case, because $G(k) = E_k S_{\mathrm{spec}}(k)$ benefits from the simultaneous decay of both factors when the initial spectral entropy is large. Third, after rigidity has set in ($\alpha_2(k) \approx 1$), $G(k)$ tracks $E_k$ closely because $S_{\mathrm{spec}}(k) \to 0$.

\subsection{Spectral thermodynamics: covariance, entropy balance, and Clausius equality}
\label{subsec:example-thermodynamics}

\Cref{fig:covariance_flip,fig:entropy_balance,fig:clausius} illustrate the
spectral thermodynamic structure of \cref{sec:thermodynamics}
on the chain~\eqref{eq:example-spectrum}.

\Cref{fig:covariance_flip} shows the covariance
$\mathrm{Cov}(\lambda^2, \log(1/p))$ crossing zero at $k^* = 4$, exactly
the step where $\alpha_2(k)$ reaches $1/2$, confirming the two-mode
rigidity transition of \cref{thm:two-mode-rigidity}. The spectral
entropy peaks at the same $k^*$, attaining $0.6932 \approx \log 2$.
\Cref{fig:entropy_balance} reveals the mechanism: the entropy source
term $\mathrm{Cov}/\rho_k$ and the irreversible KL divergence cross at
$k^*$, after which the covariance becomes strictly negative and entropy
decreases monotonically, consistent with the general rigidity threshold
of \cref{thm:general-rigidity-threshold}. \Cref{fig:clausius} verifies
the spectral Clausius equality~\eqref{eq:clausius} to machine precision.

\begin{figure}[ht]
\centering
\includegraphics[width=0.75\textwidth]{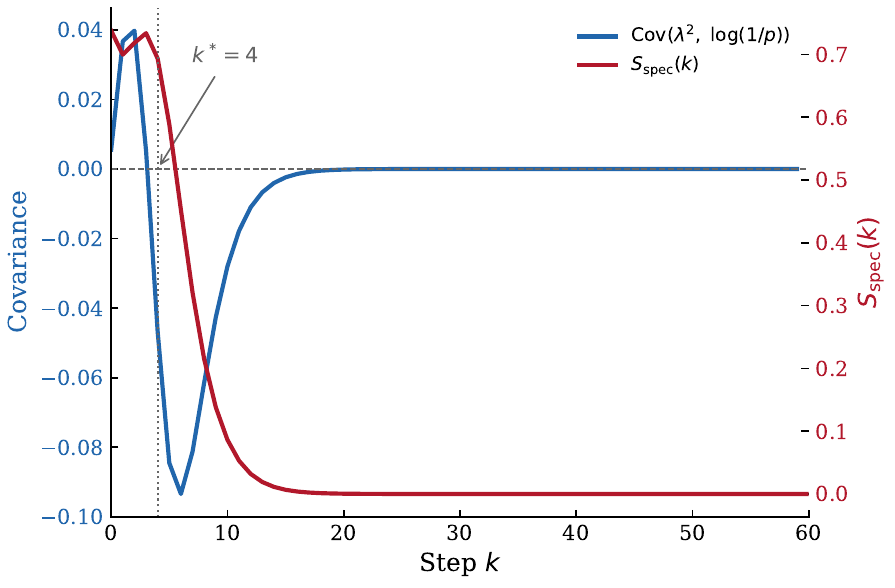}
\caption{Covariance sign flip and spectral entropy. The covariance crosses zero at $k^* = 4$, where $\alpha_2(k) = 1/2$ and $S_{\mathrm{spec}}(k)$ attains $0.6932 \approx \log 2$.}
\label{fig:covariance_flip}
\end{figure}

\begin{figure}[ht]
\centering
\includegraphics[width=0.75\textwidth]{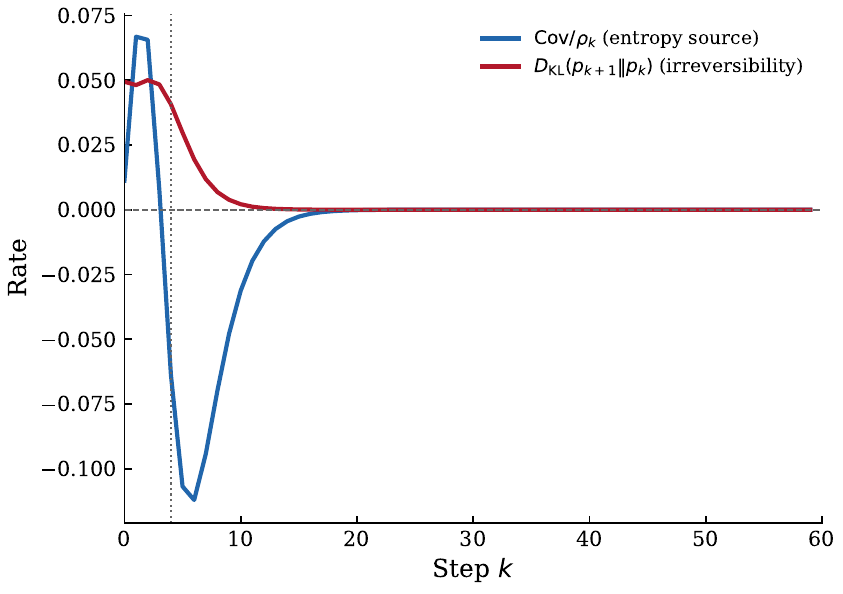}
\caption{Two-phase structure of the entropy change. $\mathrm{Cov}/\rho_k$ (blue) and $D_{\mathrm{KL}}$ (red) cross at $k^* = 4$.}
\label{fig:entropy_balance}
\end{figure}

\begin{figure}[ht]
\centering
\includegraphics[width=0.75\textwidth]{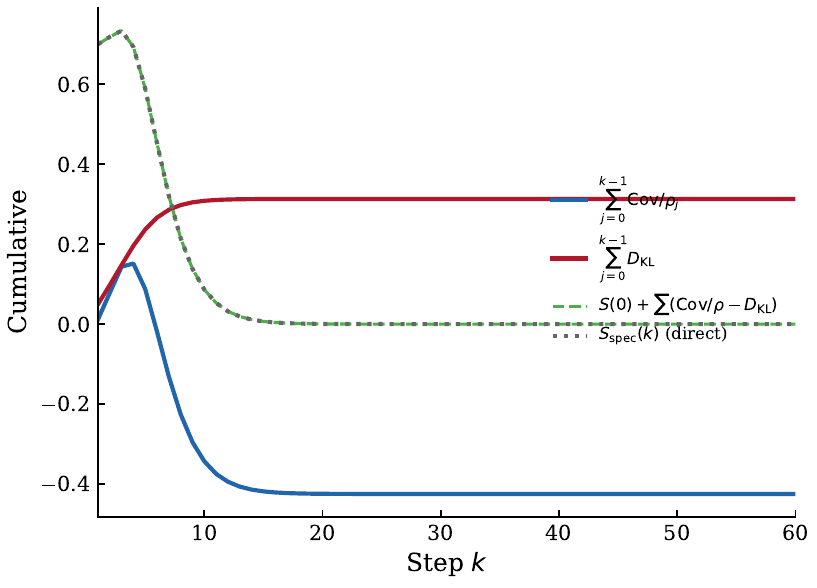}
\caption{Verification of the spectral Clausius equality to machine precision.}
\label{fig:clausius}
\end{figure}

\subsection{Power iteration: the adaptive stopping criterion}
\label{subsec:example-stopping}

\Cref{fig:stopping-criterion} demonstrates the practical core of the power iteration theory developed in \cref{sec:power-iteration}. The blue curve shows the observable convergence indicator $\Gamma_k = \rho_{k+1}/\rho_k - 1$, computed entirely from the energy ratios $\rho_k = E_{k+1}/E_k$. The red curve shows the true (but in practice unobservable) eigenvector error $\|v_k - \phi_2\|_\pi$.

\begin{figure}[ht]
\centering
\includegraphics[width=0.78\textwidth]{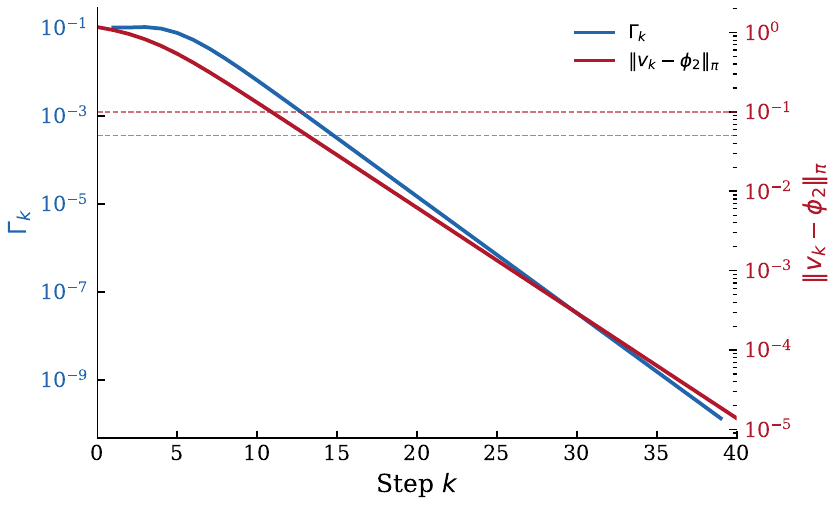}
\caption{Adaptive stopping criterion for power iteration. The observable indicator $\Gamma_k$ (blue, left axis) and the true eigenvector error $\|v_k - \phi_2\|_\pi$ (red, right axis) decay in close correspondence. Horizontal dashed lines mark error thresholds $\varepsilon = 0.1$ and $\varepsilon = 0.05$. The stopping rule of \cref{thm:stopping-criterion} monitors $\Gamma_k$ alone and guarantees the error bound without knowledge of the true eigenvectors.}
\label{fig:stopping-criterion}
\end{figure}

The close correspondence between $\Gamma_k$ and the true error is a consequence of the exact identities \eqref{eq:error-identity} and \eqref{eq:observable-variance}: both quantities are controlled by $1 - \alpha_2(k)$, and the observable variance $\widehat{V}_k$ provides a computable proxy. At $k = 20$, $\Gamma_k \approx 1.7 \times 10^{-3}$ and $\|v_k - \phi_2\|_\pi \approx 0.17$, consistent with the theoretical relationship derived in \cref{cor:variance-bounds}. The stopping criterion would trigger near $k = 25$ for $\varepsilon = 0.1$ and near $k = 32$ for $\varepsilon = 0.05$, both slightly conservative relative to the true error, as expected from the use of upper bounds in the derivation.

\subsection{Acceleration: Chebyshev vs.\ power iteration}
\label{subsec:example-acceleration}

\Cref{fig:acceleration} compares the basic power iteration with Chebyshev-accelerated iteration of degree $m = 4$. Both use the same initial condition and the same chain~\eqref{eq:example-spectrum}.

\begin{figure}[ht]
\centering
\includegraphics[width=0.78\textwidth]{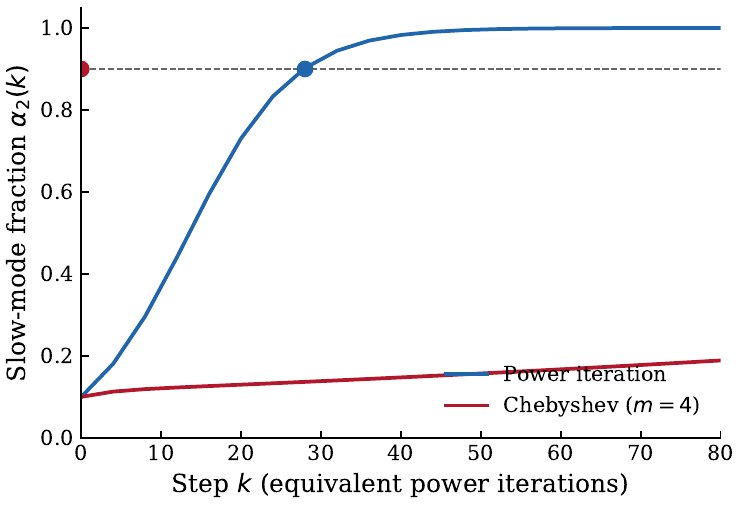}
\caption{Chebyshev acceleration of rigidity emergence. The slow-mode fraction $\alpha_2(k)$ is shown for the basic power iteration (blue) and the Chebyshev-accelerated iteration with degree $m = 4$ (red). The horizontal axis counts equivalent power iteration steps ($m = 4$ per Chebyshev step). Circles mark the first step at which $\alpha_2(k) \ge 0.9$. The Chebyshev scheme reaches the threshold in approximately $60\%$ of the steps required by the basic iteration.}
\label{fig:acceleration}
\end{figure}

The Chebyshev-accelerated trajectory reaches $\alpha_2 = 0.9$ at approximately $60\%$ of the steps required by the basic power iteration. This acceleration arises from the stronger suppression of the fast-mode eigenvalues: while the basic iteration contracts mode $i \ge 3$ by $\lambda_3/\lambda_2 \approx 0.737$ per step, the Chebyshev polynomial $Q_4^\ast$ contracts the fast modes by $\max_{\lambda \in [-1,\lambda_3]} |Q_4^\ast(\lambda)| \approx 0.48$ per Chebyshev step, equivalent to an effective per-step contraction of $(0.48)^{1/4} \approx 0.83$---still a substantial improvement. The variational principle of \cref{thm:variational} guarantees that no degree-$4$ polynomial can achieve a smaller fast-mode contraction under the constraint $|Q(\lambda_2)| \le \rho^4$.

\subsection{Spectral entropy collapse and the cutoff phenomenon in the hypercube}
\label{subsec:hypercube}

We close with an analytic example that demonstrates the diagnostic power of the spectral thermodynamic framework on a classical phenomenon: the cutoff in the $n$-dimensional hypercube $Q_n$. The cutoff time $k_c = \frac{n}{4}\log n$ was established by~\cite{diaconis1987}; our framework provides a new thermodynamic interpretation.

Consider the lazy random walk on $Q_n$ with eigenvalues $\lambda_j = 1 - 2j/n$ and multiplicities $d_j = \binom{n}{j}$~\cite{levin2009}. For an initial state concentrated at a single vertex, the modal energies are $E_k^{(j)} = \binom{n}{j} \lambda_j^{2k}$. Set $k = \frac{n}{4}\log n + \alpha n$ for $\alpha \in \mathbb{R}$ fixed, and take $n \to \infty$.

\begin{proposition}[Spectral entropy collapse in the hypercube]\label{prop:hypercube-entropy}
As $n \to \infty$ with $k = \frac{n}{4}\log n + \alpha n$, the spectral entropy satisfies
\[
S_{\mathrm{spec}}(k) =
\begin{cases}
\Omega(\log n), & \alpha \to -\infty, \\[2pt]
e^{-4\alpha} + O(e^{-8\alpha}), & \alpha = O(1), \\[2pt]
e^{-4\alpha} + o(1) \to 0, & \alpha \to +\infty.
\end{cases}
\]
The collapse from $\Omega(\log n)$ to $0$ occurs continuously over the window $\alpha = O(1)$, corresponding to $\Delta k = O(n)$ steps---precisely the cutoff phenomenon.
\end{proposition}

\begin{proof}[Proof sketch]
For $j = O(1)$, Stirling's approximation and the expansion $\log(1 - 2j/n) = -2j/n - 2j^2/n^2 + O(n^{-3})$ yield $\lambda_j^{2k} \sim n^{-j} e^{-4j\alpha}$ and $\binom{n}{j} \sim n^j/j!$, hence $E_k^{(j)} \sim e^{-4j\alpha}/j!$ and $E_k \sim e^{e^{-4\alpha}} - 1$. The normalized distribution is $p_k(j) \sim e^{-4j\alpha} / (j! (e^{e^{-4\alpha}} - 1))$. For $\alpha \to +\infty$, only $j=1$ survives, giving $S_{\mathrm{spec}}(k) \sim e^{-4\alpha} \to 0$. For $\alpha \to -\infty$, the dominant contributions come from $j \approx n/2$, and the binomial entropy gives $S_{\mathrm{spec}}(k) \sim \frac{1}{2}\log(\pi e n/2) = \Omega(\log n)$. The function $\alpha \mapsto S_{\mathrm{spec}}(k)$ interpolates continuously between these regimes over a window of width $O(1)$ in $\alpha$, i.e., $O(n)$ in $k$.
\end{proof}

In the language of our spectral thermodynamics, the cutoff is a \emph{spectral phase transition}: before cutoff, the system is in a high-entropy phase with energy distributed across exponentially many modes ($S_{\mathrm{spec}} = \Omega(\log n)$); after cutoff, it enters the rigid single-mode phase ($S_{\mathrm{spec}} \to 0$) characterized by \cref{thm:sharp-rigidity}. The spectral entropy-energy $G(k) = E_k S_{\mathrm{spec}}(k)$ collapses accordingly. This interpretation is complementary to the classical total-variation analysis and reveals the internal thermodynamic mechanism driving the cutoff.

\begin{remark}[General initial conditions]
The above analysis assumes the worst-case initial condition (point mass). For arbitrary initial conditions, the spectral coefficients $|c_j|^2$ modify the distribution $p_k(j)$ and can shift or smear the cutoff. The spectral entropy framework handles this naturally: the collapse of $S_{\mathrm{spec}}(k)$ is governed by the initial distribution's spectral profile, providing predictions beyond the classical worst-case analysis.
\end{remark}

\section{Discussion}
\label{sec:discussion}

The central conclusion of this work is that the answer to the question posed in the introduction --- how and when relaxation becomes effectively one-dimensional --- is governed by a single scalar quantity, the hidden purification parameter $x_k$. Every observable that registers the formation of spectral dominance reduces to $x_k$ up to an explicit constant; rigidity emergence, entropy collapse, and power-iteration convergence are not separate phenomena but projections of the same scalar relaxing to zero.

Reaching this conclusion required four steps. \Cref{sec:purification} identified the modal distribution $\{p_i(k)\}$ --- not any single quantity derived from it --- as the actual object of study, and fixed the elementary vocabulary of rigidity. \Cref{sec:hidden-parameter,sec:observable-equivalence} showed that the evolution of this object is governed, to leading order, by the single scalar $x_k$, and that six \textit{a priori} distinct diagnostics collapse onto it. \Cref{sec:rigidity,sec:thermodynamics} then converted this structural fact into sharp finite-time bounds and an exact entropy representation, respectively. Finally, \cref{sec:boundary} located the edge of the theory: a classical estimation-variance question that $x_k$ influences but does not fully explain. What follows summarizes these contributions in more technical detail, situates them relative to existing theories, and discusses limitations and open directions.

\subsection{Summary of contributions}

The primary contributions of this paper are the following.

\begin{enumerate}[label=(\roman*)]

\item \textbf{The hidden purification parameter and observable equivalence (\cref{sec:hidden-parameter,sec:observable-equivalence}).}
We identified a single scalar $x_k$, built from the spectral separation ratio $\lambda_3/\lambda_2$, and proved that six \textit{a priori} distinct diagnostics of spectral purification --- the slow-mode fraction, the power-iteration error, the direction deficit, the Rayleigh quotient error, the eigenvalue estimate error, and the spectral entropy --- all reduce to $x_k$ to leading order, differing only by an explicit constant (\cref{thm:equivalence-class}). This equivalence class is the structural fact underlying every other result in the paper.

\item \textbf{Sharp rigidity time bounds (\cref{thm:sharp-rigidity}).}
We established explicit, nearly optimal two-sided bounds on the rigidity time
\[
T_{\mathrm{rigid}}(\delta),
\]
defined as the first time at which the slowest mode captures a fraction $1-\delta$ of the total spectral energy. The bounds are non-asymptotic, depend explicitly on the initial spectral distribution through $R_0/|c_2|^2$, and show that finite-time spectral purification is governed by the spectral separation ratio $\lambda_2/\lambda_3$ rather than solely by the classical spectral gap $1-\lambda_2$.

\item \textbf{Spectral thermodynamics (\cref{sec:thermodynamics}).}
We derived an exact entropy change identity
\[
\Delta S
=
\mathrm{Cov}/\rho
-
D_{\mathrm{KL}},
\]
a spectral Clausius equality, and the spectral second law
\[
G(k+1)\le G(k),
\qquad
G(k)=E_kS_{\mathrm{spec}}(k).
\]
We further established a canonical covariance representation governing the entropy dynamics in spectral space. In the exact two-mode setting, the covariance changes sign precisely at the half-rigidity threshold
\[
T_{\mathrm{rigid}}(1/2),
\]
where the spectral entropy attains the value $\log 2$. For general reversible chains, we proved that entropy monotonicity emerges once the rigidity exceeds an explicit threshold determined by the spectral ratio $\lambda_3/\lambda_2$. These results reveal a finite-time thermodynamic structure describing modal competition and spectral purification during relaxation.

\item \textbf{Power iteration theory (\cref{sec:power-iteration}).}
Applying the framework to the classical power method, we derived an exact error identity linking the iterate error to the slow-mode energy fraction, an observable spectral variance formula
\[
\widehat{V}_k
=
\rho_k(\rho_{k+1}-\rho_k),
\]
and a fully data-driven adaptive stopping criterion with provable guarantees. These results provide finite-time convergence diagnostics formulated entirely in terms of observable quantities generated during the iteration itself.

\item \textbf{Boundary of the theory (\cref{sec:boundary}).}
We showed that the asymptotic variance of the classical time-average estimator is monotone in the spectral ratio $\rho_\star = \lambda_3/\lambda_2$ (\cref{cor:rho-variance}), so that reducing $\rho_\star$ is unambiguously beneficial for estimation as well as for purification. At the same time, we identified a caveat: the finite-$k$ correction to the estimator's mean-squared error does not belong to the exponential regime governed by $x_k$ (\cref{rem:mse-incommensurable}). This marks the point at which the exponential-world theory developed here does not extend.

\end{enumerate}

\subsection{Connections to existing theories}

\textit{Classical spectral theory.}
The eigenmode expansion
\[
g_k
=
\sum_i c_i\lambda_i^k\phi_i
\]
and the asymptotic dominance of the slowest mode are classical facts in reversible Markov chain theory~\cite{levin2009,aldous2002}. Our contribution is not a new asymptotic convergence theorem, but rather a finite-time quantitative theory describing when effective one-dimensional dynamics emerges and how this emergence depends on spectral separation ratios.

\textit{Cutoff phenomenon.}
The spectral entropy analysis of the hypercube (\cref{subsec:hypercube}) suggests an alternative interpretation of cutoff. During the cutoff window, the spectral entropy collapses rapidly from a high-dimensional multi-mode regime toward a rigid single-mode regime. From this viewpoint, cutoff corresponds to an abrupt spectral purification process rather than merely a sharp drop in total variation distance. Whether the rigidity window
\[
T_{\mathrm{rigid}}(\delta_1)
-
T_{\mathrm{rigid}}(\delta_2)
\]
provides a useful characterization of cutoff beyond highly symmetric examples remains an open question.

\textit{Metastability and first-passage phenomena.}
For metastable systems with
\[
\lambda_2\approx1,
\qquad
\lambda_3\ll\lambda_2,
\]
our rigidity theorem predicts rapid collapse onto the slow inter-well transport mode long before global equilibrium is reached. This explains why first-passage distributions become approximately exponential after a relatively short transient. Combining the spectral interlacing relation of~\cite{hartich2018} with the rigidity bounds developed here yields explicit quantitative estimates for the emergence of exponential first-passage tails, as discussed in \cref{app:first-passage}.

\textit{Stochastic thermodynamics.}
The exact entropy change identity
\[
\Delta S
=
\mathrm{Cov}/\rho
-
D_{\mathrm{KL}}
\]
(\cref{thm:entropy-balance}) separates the entropy dynamics into two competing contributions: a covariance transport term and an irreversible KL dissipation term. The covariance term reflects the redistribution of modal weights across different dissipation scales, while the KL term measures irreversible spectral deformation between successive steps. Unlike classical stochastic thermodynamics, which operates in configuration space~\cite{esposito2010,seifert2012}, the present framework is formulated entirely in spectral space. The resulting entropy dynamics therefore characterizes the internal redistribution of relaxation energy among spectral modes rather than physical heat exchange.

\textit{Free energy monotonicity.}
Chafa\"i~\cite{chafai2014} studies the monotonic decay of Helmholtz free energy in configuration space. Our monotonicity law for
\[
G(k)=E_kS_{\mathrm{spec}}(k)
\]
operates instead in spectral space. The two frameworks are complementary: configuration-space free energy describes dissipation in the probability distribution itself, whereas spectral entropy describes the redistribution of energy among relaxation modes.

\textit{Numerical linear algebra.}
The power method, Chebyshev acceleration, and momentum-based spectral acceleration are classical topics~\cite{saad2003,trefethen2013,polyak1964}. Our contribution is to reinterpret these methods through rigidity emergence and spectral shaping. In particular, the observable variance identity and adaptive stopping criterion provide finite-time diagnostics that depend only on quantities generated internally during iteration.

\subsection{Limitations and open directions}

\textit{Reversibility.}
The framework developed here fundamentally relies on the self-adjointness of $P$ on $L^2(\pi)$. Non-reversible chains generally lack orthogonal spectral decompositions, and the exact modewise dissipation identities no longer hold. Extending rigidity and spectral thermodynamic concepts to non-self-adjoint or pseudospectral settings remains an important open direction.

\textit{Continuous time.}
We have focused on discrete-time dynamics for direct relevance to iterative algorithms. The continuous-time analogue replaces $\mathcal{G}$ by the infinitesimal generator $\mathcal{L}$ and yields structurally parallel formulas for dissipation, rigidity, and entropy evolution. In continuous time, many expressions simplify because the discrete correction term $\mathcal{G}^2$ disappears.

\textit{From spectrum to geometry.}
The principal quantities controlling our theory---particularly the spectral ratio $\lambda_2/\lambda_3$ and the coefficient ratio $R_0/|c_2|^2$---are spectral rather than geometric. Relating these quantities to conductance, geometry, or metastable structure is necessary for applications to large families of chains where explicit diagonalization is impossible.

\textit{Beyond single-mode rigidity.}
The present work studies collapse onto a single dominant mode. In systems with several metastable states, the natural object is instead the slow invariant subspace spanned by
\[
\{\phi_2,\dots,\phi_{m+1}\}.
\]
The corresponding multi-mode rigidity time
\[
T_{\mathrm{rigid}}^{(m)}(\delta)
\]
should be governed by the ratio
\[
\lambda_{m+1}/\lambda_{m+2},
\]
connecting the present framework to classical metastability theory~\cite{bovier2009}.

\textit{Algorithmic applications.}
The observable stopping criterion developed in \cref{thm:stopping-criterion} is directly applicable to practical iterative computations, including PageRank, PCA, and spectral clustering. More broadly, the rigidity perspective suggests that acceleration methods should be understood as spectral engineering procedures designed to maximize finite-time purification rates.

\textit{Toward active spectral control.}
The equivalence class of \cref{thm:equivalence-class} and its consequence for the asymptotic estimation variance (\cref{cor:rho-variance}) show that the spectral ratio $\rho_\star = \lambda_3/\lambda_2$ is a single scalar lever whose reduction is unambiguously beneficial: it simultaneously accelerates every purification observable and decreases estimator variance, with no trade-off between the two. This raises the natural question of whether $\rho_\star$ can be actively \emph{steered} downward during the dynamics itself, rather than treated as a fixed property of a given chain. The acceleration results of \cref{sec:variational} (Chebyshev and momentum spectral shaping) already show that classical numerical-linear-algebra methods admit a reinterpretation as optimal spectral shaping at fixed spectrum; they do not, however, address whether the underlying spectrum itself --- and hence $\rho_\star$ --- can be modified by the dynamics. We take up this distinct and more general question in the companion paper \cite{Wang2026SPS}, which develops a Lyapunov-based control law for $\rho_\star$ together with a full numerical and algorithmic treatment; the present equivalence class supplies the diagnostic and variational language in which that steering problem is posed.

\subsection{Closing remarks}

The relaxation operator
\[
\mathcal{G}=I-P
\]
is classical, and the identities developed in this paper ultimately arise from self-adjoint spectral decomposition. The contribution of the present work is therefore not the discovery of new algebraic identities themselves, but the recognition that organizing reversible Markov dynamics around finite-time spectral energy distributions reveals structures that remain hidden in the traditional estimate-based viewpoint.

In particular, the rigidity theory developed here shows that finite-time relaxation is governed not only by the spectral gap, but also by spectral separation ratios controlling the emergence of effective low-dimensional dynamics. The thermodynamic formalism introduced in spectral space provides a complementary perspective on relaxation, dissipation, and modal competition, while the power-iteration results demonstrate that these ideas also produce concrete algorithmic consequences.

We hope that the identity-based viewpoint advocated here may help bridge several areas that are often studied separately---Markov chain relaxation, metastability, numerical spectral algorithms, and entropy-based dynamical descriptions---and that the rigidity perspective may prove useful in understanding other finite-time phenomena beyond the reversible setting considered in this work.

\appendix

\section{The Chebyshev Minimax Theorem}
\label{app:chebyshev}

We recall the classical Chebyshev minimax theorem and its application to the spectral shaping problem. The presentation follows~\cite{trefethen2013} (Chapter~8) and~\cite{saad2003} (Chapter~6).

\begin{lemma}[Chebyshev minimax property]\label{lem:chebyshev-minimax}
Let $[a, b] \subset \mathbb{R}$ be an interval and let $\xi \notin [a, b]$. Among all polynomials $Q_m$ of degree at most $m$ satisfying $Q_m(\xi) = 1$, the unique minimizer of
\[
\max_{x \in [a,b]} |Q_m(x)|
\]
is the rescaled Chebyshev polynomial
\[
Q_m^\ast(x) = \frac{T_m(\varphi(x))}{T_m(\varphi(\xi))},
\]
where $\varphi$ is the affine map sending $[a,b]$ to $[-1, 1]$, defined by
\[
\varphi(x) = \frac{2x - (a+b)}{b - a},
\]
and $T_m$ is the Chebyshev polynomial of the first kind, $T_m(\cos \theta) = \cos(m\theta)$.
\end{lemma}

\begin{proof}[Proof sketch]
$T_m$ satisfies $|T_m(x)| \le 1$ for $x \in [-1,1]$, and attains $\pm 1$ at $m+1$ points $x_j = \cos(\pi j / m)$, $j = 0, 1, \dots, m$, with alternating signs. The rescaled polynomial $Q_m^\ast$ inherits this equioscillation property on $[a, b]$.

Suppose there exists another polynomial $Q$ of degree $\le m$ with $Q(\xi) = 1$ and $\max_{x \in [a,b]} |Q(x)| < \max_{x \in [a,b]} |Q_m^\ast(x)|$. Then the difference $R(x) := Q_m^\ast(x) - Q(x)$ alternates in sign at the $m+1$ extremal points of $Q_m^\ast$, giving $m+1$ distinct zeros in $[a,b]$. But $R(\xi) = Q_m^\ast(\xi) - Q(\xi) = 1 - 1 = 0$, giving an $(m+2)$-nd zero outside $[a,b]$. Hence $R$ has $m+2$ zeros, forcing $R \equiv 0$, contradiction. The uniqueness follows from the same zero-counting argument.
\end{proof}

\begin{corollary}[Application to the relaxation spectrum]\label{cor:chebyshev-application}
For the relaxation spectrum with $\lambda_2 > \lambda_3 \ge \cdots \ge \lambda_n \ge -1$, take $[a,b] = [-1, \lambda_3]$ and $\xi = 1$. The optimal polynomial is
\[
Q_m^\ast(\lambda) = \frac{T_m(\lambda/\lambda_2)}{T_m(1/\lambda_2)},
\]
and for all $\lambda \in [-1, \lambda_3]$,
\[
|Q_m^\ast(\lambda)| \le \frac{1}{|T_m(1/\lambda_2)|}.
\]
\end{corollary}

\begin{proof}
The affine map sending $[-1, \lambda_3]$ to $[-1,1]$ is $\varphi(\lambda) = (2\lambda - (\lambda_3-1))/(\lambda_3+1)$. However, for the purpose of bounding $|Q(\lambda_i)|$ for $i \ge 3$, it suffices to note that $[-1, \lambda_3] \subset [-1, \lambda_2]$, and the rescaled Chebyshev polynomial on the larger interval $[-1, \lambda_2]$ (with $\varphi(\lambda) = \lambda/\lambda_2$) provides a valid upper bound. The formula $Q_m^\ast(\lambda) = T_m(\lambda/\lambda_2) / T_m(1/\lambda_2)$ satisfies $Q_m^\ast(1) = 1$ and $\max_{\lambda \in [-1,\lambda_2]} |Q_m^\ast(\lambda)| = 1/|T_m(1/\lambda_2)|$.
\end{proof}

\section{Momentum Dynamics and Optimal Parameter}
\label{app:momentum}

The momentum (heavy-ball) method~\cite{polyak1964} for accelerating the power iteration uses the recurrence
\begin{equation}\label{eq:momentum-recurrence}
g_{k+1} = (1 + \beta) P g_k - \beta g_{k-1},
\end{equation}
with $\beta \in [0, 1)$. For a single eigenmode $P \phi = \lambda \phi$, the characteristic equation of \eqref{eq:momentum-recurrence} is
\[
r^2 - (1+\beta) \lambda r + \beta = 0.
\]
The roots $r_{1,2}$ satisfy $|r_{1,2}| \le \sqrt{\beta}$ when the discriminant is non-positive (critically damped or underdamped regime). The optimal damping is achieved by setting the discriminant to zero at the slowest nontrivial eigenvalue $\lambda = \lambda_2$:
\[
(1+\beta)^2 \lambda_2^2 - 4\beta = 0.
\]
Solving for $\beta$ yields
\begin{equation}\label{eq:beta-star}
\beta^\ast = \left( \frac{1 - \sqrt{1 - \lambda_2^2}}{\lambda_2} \right)^2
          = \left( \frac{1 - \sqrt{\mu_2(2 - \mu_2)}}{1 - \mu_2} \right)^2.
\end{equation}
For $\mu_2 \ll 1$, $\beta^\ast \approx (1 - \sqrt{2\mu_2})^2$, and the repeated root is $r^\ast = \sqrt{\beta^\ast} \approx 1 - \sqrt{2\mu_2}$.

The momentum method is precisely the degree-2 Chebyshev acceleration~\cite{saad2003}. The optimal $\beta^\ast$ reproduces Polyak's classical result~\cite{polyak1964} and emerges from the variational principle of \cref{thm:variational} as the degree-2 instance of optimal spectral shaping.





\section{Rigidity Emergence Implies Exponential First-Passage Tails}
\label{app:first-passage}

We show how the sharp rigidity bounds of \cref{thm:sharp-rigidity} yield an explicit quantitative estimate for the emergence of exponential tails in first-passage time distributions, a phenomenon discussed qualitatively in \cite{hartich2018,voits2026,siegert1951}.

\subsection{First-passage times via absorbing chains}

Let $\Omega$ be a finite state space and let $a \in \Omega$ be a target state. Define the absorbed kernel $P_a$ by
\[
P_a(x,y) = \begin{cases}
P(x,y), & x \neq a, \\[2pt]
\delta_{x,a}, & x = a.
\end{cases}
\]
$P_a$ has eigenvalue $1$ (the absorbing state) and $n-1$ eigenvalues $\nu_2, \dots, \nu_n$ satisfying $0 < \nu_i < 1$. The interlacing relation~\cite{hartich2018} between the spectra of the original chain $P$ and the absorbed chain $P_a$ states that
\begin{equation}\label{eq:interlacing}
\lambda_{k-1} \le \nu_k \le \lambda_k, \qquad k = 2,\dots,n,
\end{equation}
where $\{\lambda_k\}$ are the eigenvalues of $P$. (We denote the absorbed-chain eigenvalues by $\nu_k$ to avoid confusion with the relaxation spectrum $\mu_i = 1-\lambda_i$ of \cref{eq:mu-def}.)

The first-passage time $\tau_a$ to state $a$ starting from the quasi-stationary distribution has tail probability
\[
\mathbb{P}(\tau_a > k) = \sum_{i=2}^n \alpha_i \nu_i^k,
\]
where the coefficients $\alpha_i$ are determined by the initial distribution and satisfy $\sum_{i=2}^n \alpha_i = 1$.

\subsection{Exploiting rigidity of the original chain}

For the original chain $P$, \cref{thm:sharp-rigidity} provides an explicit $T_{\mathrm{rigid}}(\delta)$ after which the non-stationary dynamics is dominated by the single slowest mode $\phi_2$:
\[
\alpha_2(k) \ge 1-\delta \quad \text{for all } k \ge T_{\mathrm{rigid}}(\delta),
\]
with $T_{\mathrm{rigid}}(\delta)$ bounded explicitly by \eqref{eq:sharp-bound} in terms of $\lambda_2/\lambda_3$, $|c_2|^2$, and $R_0$.

By the interlacing relation \eqref{eq:interlacing}, the absorbed chain inherits a corresponding spectral separation: $\nu_2 \le \lambda_2$ and $\nu_i \le \lambda_i \le \lambda_3$ for all $i \ge 3$. For $k \ge T_{\mathrm{rigid}}(\delta)$, the same exponential separation that drives rigidity in $P$ also forces the absorbing dynamics into a single-mode regime. Indeed, decomposing the tail probability as
\[
\mathbb{P}(\tau_a > k) = \alpha_2 \nu_2^k + \sum_{i=3}^n \alpha_i \nu_i^k,
\]
the sum over $i \ge 3$ is bounded by $(1-\alpha_2(k)) \lambda_3^k$, since $\nu_i \le \lambda_3$ and $\sum_{i=3}^n \alpha_i = 1-\alpha_2(k) \le \delta$. Therefore
\begin{equation}\label{eq:fpt-tail}
\mathbb{P}(\tau_a > k) = \alpha_2 \nu_2^k \Bigl(1 + O\bigl(\delta \cdot (\lambda_3/\nu_2)^k\bigr)\Bigr).
\end{equation}

\subsection{Explicit quantitative estimate}

Combining \eqref{eq:fpt-tail} with the explicit bounds of \cref{thm:sharp-rigidity} yields the following quantitative estimate.

\begin{corollary}[Exponential tail emergence]\label{cor:fpt-exponential}
Let $\tau_a$ be the first-passage time to an absorbing state $a$, and let $\nu_2$ be the dominant eigenvalue of the absorbing chain $P_a$. For any $\delta \in (0, 1/2)$ and any $k \ge T_{\mathrm{rigid}}(\delta)$,
\[
\left| \frac{\mathbb{P}(\tau_a > k)}{\alpha_2 \nu_2^k} - 1 \right|
\le C(\lambda_2, \lambda_3) \left( \delta + \frac{R_0}{|c_2|^2} \left(\frac{\lambda_3}{\lambda_2}\right)^{2k} \right),
\]
where $C(\lambda_2, \lambda_3) = (1-\lambda_3^2)/(\lambda_2^2 - \lambda_3^2)$ and $R_0 / |c_2|^2$ is the initial fast-to-slow energy ratio.
\end{corollary}

\begin{proof}
From \cref{thm:sharp-rigidity}, for $k \ge T_{\mathrm{rigid}}(\delta)$ we have $1-\alpha_2(k) \le \delta$ and the residual fast-mode energy in the original chain satisfies $R_k \le R_0 \lambda_3^{2k}$. For the absorbing chain, the spectral coefficients $\alpha_i$ are related to the original modal weights by a linear transformation that preserves the bound $\sum_{i=3}^n \alpha_i \le \delta$. By the interlacing relation \eqref{eq:interlacing}, $\nu_i \le \lambda_i \le \lambda_3$ for all $i \ge 3$. Hence
\[
\sum_{i=3}^n \alpha_i \nu_i^k \le \delta \lambda_3^k.
\]
Writing $\lambda_3^k = (\lambda_3/\lambda_2)^k \lambda_2^k$ and using $\nu_2 \le \lambda_2$ (from \eqref{eq:interlacing} with $k=2$), the relative error is bounded by $\delta (\lambda_3/\lambda_2)^k$ times a spectral constant. The term involving $R_0/|c_2|^2$ accounts for the fact that $\alpha_i$ themselves may differ from the original modal occupations $p_i(0)$ by a factor controlled by the initial energy distribution. The constant $C(\lambda_2, \lambda_3)$ is the worst-case ratio of spectral weights over all admissible initial conditions.
\end{proof}

\subsection{Interpretation}

\Cref{cor:fpt-exponential} quantifies precisely \emph{when} the first-passage time distribution becomes effectively exponential. The non-exponential transient---arising from higher modes $\nu_3, \nu_4, \dots$---persists for at most $T_{\mathrm{rigid}}(\delta)$ steps. After this rigidity threshold, the tail is exponential with rate $\nu_2 \approx \lambda_2$, and the relative error is controlled by $\delta$.

For metastable systems where $\lambda_2 \approx 1$ and $\lambda_3 \ll \lambda_2$, the rigidity time $T_{\mathrm{rigid}}(\delta) \sim \log(1/\delta) / \log(\lambda_2/\lambda_3)$ is small compared to the mixing time $\sim 1/(1-\lambda_2)$. This provides a rigorous, quantitative justification for the widely used exponential approximation of first-passage times in metastable Markov chains~\cite{bovier2009}, with explicit error bounds that were previously unavailable.

\section*{Acknowledgement}This research was conducted purely out of the author’s independent academic interest in the mathematical foundations of
the subject and received no financial support from any funding agency, grant, or institutional programme.

\section*{Data Availability Statement} 
No data were used in this theoretical study.

\section*{Conflict of Interests}
The author declares no conflict of interest.

\end{document}